\documentclass{article}
\usepackage{amsmath}
\usepackage{amsthm}
\usepackage{amssymb}
\usepackage[ansinew]{inputenc}
\usepackage{url}
\usepackage[all,dvips,arc,curve,color,frame]{xy}
\input xy
\xyoption{all}

\newtheorem{lemma}{Lemma}[section]

\newtheorem{theorem}[lemma]{Theorem}
\newtheorem{corollary}[lemma]{Corollary}

\newtheorem{prop}[lemma]{Proposition}
\newtheorem{defi}[lemma]{Definition}

\newtheorem{rem}[lemma]{Remark}
\newtheorem{nota}[lemma]{Notation}

\DeclareMathOperator\ima{im}
\DeclareMathOperator\rank{rank}

\DeclareMathOperator\id{id}

\DeclareMathOperator\fend{End}

\DeclareMathOperator\Ker{ker}
\DeclareMathOperator\sym{Sym}

\DeclareMathOperator\kd{kd}

\newcommand{\ra}{\rightarrow}

\newcommand{\imp}{\Rightarrow}

\newcommand{\ran}{\rangle}

\newcommand{\lam}{\lambda}

\newcommand{\cg}{\mathcal{G}}
\newcommand{\te}{\mbox{\tiny $E$}}
\newcommand{\cge}{\mathcal{G}_{\te}}
\newcommand\ara{\stackrel{a}{\ra}}
\newcommand\arb{\stackrel{b}{\ra}}

\newcommand\mrl{\mathit{rl}}

\title{Minimal Paths in the Commuting Graphs of Semigroups}

\author{Jo\~ao Ara\'ujo\footnote{Corresponding author.}\\
  {\small Universidade Aberta, R. Escola Polit\'{e}cnica, 147}\\
  {\small 1269-001 Lisboa, Portugal}\\{\footnotesize \&}\\
  {\small Centro de \'{A}lgebra, Universidade de Lisboa}\\
  {\small 1649-003 Lisboa, Portugal, jaraujo@ptmat.fc.ul.pt}\\
{\small (+351)21 790 47 00 \ Fax (+351) 21 795 42 88}
\and Michael Kinyon\\
{\small Department of Mathematics, University of Denver}\\
{\small Denver, Colorado 80208, mkinyon@math.du.edu}
\and Janusz Konieczny\\
{\small Department of Mathematics, University of Mary Washington}\\
{\small Fredericksburg, Virginia 22401, jkoniecz@umw.edu}}
\date{}
\begin{document}
\maketitle

\begin{abstract}
Let $S$ be a finite non-commutative semigroup. The commuting graph of $S$, denoted $\cg(S)$,
is the graph whose vertices are the non-central elements of $S$ and whose edges are the sets
$\{a,b\}$ of vertices such that $a\ne b$ and $ab=ba$. Denote by $T(X)$ the semigroup of full transformations on a finite set $X$.
Let $J$ be any ideal of $T(X)$ such that $J$ is different from the ideal of constant transformations
on $X$. We prove that if $|X|\geq4$, then, with a few exceptions, the diameter of $\cg(J)$ is $5$.
On the other hand, we prove that for every positive integer $n$, there exists a semigroup $S$ such that the  diameter of $\cg(S)$ is $n$.

We also study the left paths in $\cg(S)$, that is, paths $a_1-a_2-\cdots-a_m$ such that
$a_1\ne a_m$ and $a_1a_i=a_ma_i$ for all $i\in \{1,\ldots ,m\}$.
We prove
that for every positive integer $n\geq2$, except $n=3$, there exists a semigroup
whose shortest left path has length $n$.
As a corollary, we use the previous results to solve  a purely algebraic old problem  posed by B.M. Schein.

\vskip 2mm

\noindent\emph{$2010$ Mathematics Subject Classification\/}. 05C25, 05C12, 20M20.

\vskip 2mm


\end{abstract}

\section{Introduction}
\setcounter{equation}{0}
The commuting graph of a finite non-abelian group $G$ is a simple graph whose
vertices are all non-central elements of $G$ and two distinct vertices $x,y$ are adjacent if $xy=yx$.
Commuting graphs of various groups have been studied in terms of their properties (such as connectivity or diameter),
for example in \cite{BaBu03}, \cite{Bu06}, \cite{IrJa08}, and \cite{Se01}.
They have also been used as a tool to prove group theoretic results, for example in \cite{Be83}, \cite{RaSe01},
and \cite{RaSe02}.

The concept of the commuting graph carries over to semigroups. Let $S$ be a finite non-commutative semigroup
with center $Z(S)=\{a\in S:ab=ba\mbox{ for all $b\in S$}\}$. The \emph{commuting graph}
of $S$, denoted $\cg(S)$, is the simple graph (that is, an undirected graph with no multiple edges or loops)
whose vertices are the elements of $S-Z(S)$ and whose edges are the sets $\{a,b\}$ such
that $a$ and $b$ are distinct vertices with $ab=ba$.

This paper initiates the study of commuting graphs of semigroups. Our main goal is to study
the lengths of minimal paths. We shall consider two types of paths: ordinary paths and so called left paths.

We first investigate the semigroup $T(X)$
of full transformations on a finite set $X$, and determine the diameter of
the commuting graph of every ideal of $T(X)$ (Section~\ref{stx}). We find that, with a few exceptions,
the diameter of $\cg(J)$, where $J$ is an ideal of $T(X)$, is $5$.
This small diameter does not extend to semigroups in general. We prove that for every $n\geq2$,
there is a finite semigroup $S$ whose commuting graph has diameter $n$ (Theorem~\ref{tald}).
To prove the existence of such a semigroup, we use our work on the {\em left paths} in the commuting graph of a semigroup.

Let $S$ be a semigroup. A path $a_1-a_2-\cdots-a_m$ in $\cg(S)$ is called a \emph{left path} (or $l$-path) if $a_1\ne a_m$ and
$a_1a_i=a_ma_i$ for every $i\in\{1,\ldots,m\}$. If there is any $l$-path in $\cg(S)$,
we define the \emph{knit degree} of $S$, denoted $\kd(S)$, to be the length of a shortest $l$-path in $\cg(S)$.

For every $n\geq2$ with $n\ne3$, we construct a band (semigroup of idempotents) of knit degree $n$ (Section~\ref{smlp}).
It is an open problem if there is a semigroup of knit degree $3$. The constructions presented in Section~\ref{smlp}
also give a band $S$ whose commuting graph has diameter $n$ (for every $n\geq4$).
As another application of our work on the left paths, we settle a conjecture on bands formulated by B.M. Schein in 1978 (Section~\ref{ssch}). Finally, we present some problems regarding the commuting graphs of semigroups (Section~\ref{spro}).

\section{Commuting Graphs of Ideals of $T(X)$}\label{stx}
\setcounter{equation}{0}

Let $T(X)$ be the semigroup of full transformations on a finite set $X$, that is,
the set of all functions from $X$ to $X$ with composition as the operation.
We will write functions on the right and compose from left to right, that is,
for $a,b\in T(X)$ and $x\in X$, we will write $xa$ (not $a(x)$)
and $x(ab)=(xa)b$ (not $(ba)(x)=b(a(x))$).
In this section, we determine
the diameter of the commuting graph of every ideal of $T(X)$.
Throughout this section, we assume that $X=\{1,\ldots,n\}$.

Let $\Gamma$ be a simple graph, that is, $\Gamma=(V,E)$, where $V$ is a finite
non-empty set of vertices and $E\subseteq\{\{u,v\}:u,v\in V, u\ne v\}$ is a set of edges.
We will write $u-v$ to mean that $\{u,v\}\in E$. Let $u,w\in V$. A \emph{path} in $\Gamma$ from $u$ to $w$
is a sequence of pairwise distinct vertices
$u=v_1,v_2,\ldots,v_m=w$ ($m\geq1$) such that $v_i-v_{i+1}$ for every $i\in\{1,\ldots,m-1\}$.
If $\lam$ is a path $v_1,v_2,\ldots,v_m$,
we will write $\lam=v_1-v_2-\cdots-v_m$ and say that $\lam$ has \emph{length} $m-1$.
We say that a path $\lam$ from $u$ to $w$
is a \emph{minimal path} if there is no path from $u$ to $w$ that is shorter than $\lam$.

We say that the \emph{distance} between vertices $u$ and $w$ is $k$, and write $d(u,w)=k$,
if a minimal path from $u$ to $w$ has length $k$. If there is no path from $u$ to $w$, we
say that the distance between $u$ and $w$ is infinity, and write $d(u,w)=\infty$. The maximum
distance $\max\{d(u,w):u,w\in V\}$ between vertices of $\Gamma$ is called the \emph{diameter}
of $\Gamma$. Note that the diameter of $\Gamma$ is finite if and only if $\Gamma$ is connected.

If $S$ is a finite non-commutative semigroup, then the commuting graph $\cg(S)$
is a simple graph with $V=S-Z(S)$ and, for $a,b\in V$, $a-b$ if and only if $a\ne b$ and $ab=ba$.

For $a\in T(X)$, we denote by $\ima(a)$ the image of $a$, by $\Ker(a)=\{(x,y)\in X\times X:xa=ya\}$ the kernel of $a$,
and by $\rank(a)=|\ima(a)|$ the rank of $a$.
It is well known (see \cite[Section~2.2]{ClPr64}) that in $T(X)$ the only element of $Z(T(X))$
is the identity transformation on $X$, and that $T(X)$ has exactly $n$ ideals: $J_1,J_2,\ldots,J_n$, where, for $1\leq r\leq n$,
\[
J_r=\{a\in T(X):\rank(a)\leq r\}.
\]
Each ideal $J_r$ is principal and any $a\in T(X)$ of rank $r$ generates $J_r$.
The ideal $J_1$ consists of the transformations of rank $1$ (that is, constant transformations),
and it is clear that $\cg(J_1)$ is the graph with $n$ isolated vertices.

Let $S$ be a semigroup. We denote by $\cge(S)$ the subgraph of $\cg(S)$ induced by the
non-central idempotents of $S$. The graph $\cge(S)$ is said to be the \emph{idempotent commuting graph} of $S$.
We first determine the diameter of $\cge(J_r)$. This approach is justified by the following lemma.

\begin{lemma}\label{lad1}
Let $2\leq r<n$ and let $a,b\in J_r$ be such that $ab\ne ba$. Suppose
$a-a_1-a_2-\cdots-a_k-b$ ($k\geq1$) is a minimal path in $\cg(J_r)$ from $a$ to $b$.
Then there are idempotents $e_1,e_2,\ldots,e_k\in J_r$
such that $a-e_1-e_2-\cdots-e_k-b$ is a minimal path in $\cg(J_r)$ from $a$ to $b$.
\end{lemma}
\begin{proof}
Since $J_r$ is finite, there is an integer $p\geq1$ such that $e_1=a_1^p$ is an idempotent in $J_r$.
Note that $e_1\notin Z(J_r)$ since for any $x\in X-\ima(e_1)$, $e_1$ does not commute with $c_x\in J_r$, where $c_x$ is the constant
transformation with $\ima(c_x)=\{x\}$.
Since $a_1$ commutes with $a$ and $a_2$, the idempotent $e_1=a_1^p$ also commutes with $a$ and $a_2$, and so
$a-e_1-a_2-\cdots-a_k-b$. Repeating the foregoing argument for $a_2,\ldots,a_k$,
we obtain idempotents $e_2,\ldots,e_k$ in $J_r$ such that $a-e_1-e_2-\cdots-e_k-b$.
Since  the path $a-a_1-a_2-\cdots-a_k-b$ is minimal, it follows that $a,e_1,e_2,\ldots,e_k,b$ are pairwise distinct
and the path $a-e_1-e_2-\cdots-e_k-b$ is minimal.
\end{proof}

It follows from Lemma~\ref{lad1} that if $d$ is the diameter of $\cge(J_r)$, then the diameter of $\cg(J_r)$
is at most $d+2$.

\subsection{Idempotent Commuting Graphs}\label{sside}
In this subsection,
we assume that $n\geq3$ and $2\leq r<n$.
We will show that, with some exceptions, the diameter of $\cge(J_r)$ is $3$ (Theorem~\ref{tdia}).

Let $e\in T(X)$ be an idempotent. Then there is a unique partition $\{A_1,A_2,\ldots,A_k\}$ of $X$ and
unique elements $x_1\in A_1, x_2\in A_2,\ldots, x_k\in A_k$ such that for every $i$,
$A_ie=\{x_i\}$. The partition $\{A_1\ldots,A_k\}$ is induced by the kernel of $e$, and $\{x_1,\ldots,x_k\}$
is the image of $e$. We will use the following notation for $e$:
\begin{equation}\label{eide}
e=(A_1,x_1\ran(A_2,x_2\ran\ldots(A_k,x_k\ran.
\end{equation}
Note that $(X,x\ran$ is the constant idempotent with image $\{x\}$. The following result has been obtained in \cite{ArKo03} and
\cite{Ko02} (see also \cite{ArKo04}).

\begin{lemma}\label{lcen}
Let $e=(A_1,x_1\ran(A_2,x_2\ran\ldots(A_k,x_k\ran$ be an idempotent in $T(X)$ and let $b\in T(X)$. Then $b$ commutes with $e$
if and only if for every $i\in\{1,\ldots,k\}$, there is $j\in\{1,\ldots,k\}$ such that
$x_ib=x_j$ and $A_ib\subseteq A_j$.
\end{lemma}
We will use Lemma~\ref{lcen} frequently, not always mentioning it explicitly.
The following lemma is an immediate consequence of Lemma~\ref{lcen}.

\begin{lemma}\label{ljo1}
Let $e,f\in J_r$ be idempotents and suppose
there is $x\in X$ such that $x\in\ima(e)\cap\ima(f)$. Then $e-(X,x\ran-f$.
\end{lemma}

\begin{lemma}\label{ljo2}
Let $e,f\in J_r$ be idempotents such that $\ima(e)\cap\ima(f)=\emptyset$.
Suppose there is $(x,y)\in\ima(e)\times\ima(f)$ such that $(x,y)\in\Ker(e)\cap\Ker(f)$.
Then there is an idempotent $g\in J_r$ such that $e-g-f$.
\end{lemma}
\begin{proof}
Let $e=(A_1,x_1\ran\ldots(A_k,x_k\ran$ and $f=(B_1,y_1\ran\ldots(B_m,y_m\ran$. We may assume
that $x=x_1$ and $y=y_1$. Since $(x,y)\in\Ker(e)\cap\Ker(f)$, we have $y\in A_1$ and $x\in B_1$.
Let $g=(\ima(e),x\ran(X-\ima(e),y\ran$. Then $g$ is
in $J_r$ since $\rank(g)=2$ and $r\geq2$. By Lemma~\ref{lcen}, we have $eg=ge$
(since $y\in A_1$) and $fg=gf$ (since $\ima(f)\subseteq X-\ima(e)$ and $x\in B_1$). Hence $e-g-f$.
\end{proof}

\begin{lemma}\label{lja1}
Let $e,f\in J_r$ be idempotents such that $\ima(e)\cap\ima(f)=\emptyset$. Then there are
idempotents $g,h\in J_r$ such that $e-g-h-f$.
\end{lemma}
\begin{proof}
Let $e=(A_1,x_1\ran\ldots(A_k,x_k\ran$ and $f=(B_1,y_1\ran\ldots(B_m,y_m\ran$. Since $\{A_1,\ldots,A_k\}$
is a partition of $X$,
there is $i$ such that $y_1\in A_i$. We may assume that $y_1\in A_1$. Let $g=(X-\{y_1\},x_1\ran(\{y_1\},y_1\ran$
and $h=(X,y_1\ran$. Then $g$ and $h$ are in $J_r$ (since $r\geq2$).
By Lemma~\ref{lcen}, $eg=ge$, $gh=hg$, and $hf=fh$. Thus $e-g-h-f$.
\end{proof}

\begin{lemma}\label{lja2}
Let $m$ be a positive integer such that $2m\leq n$, $\sigma$ be an $m$-cycle on $\{1,\ldots,m\}$, and
\[
e=(A_1,x_1\ran(A_2,x_2\ran\ldots(A_m,x_m\ran\mbox{ and }f=(B_1,y_1\ran(B_2,y_2\ran\ldots(B_m,y_m\ran
\]
be idempotents in $T(X)$ such that $x_1,\ldots,x_m,y_1,\ldots,y_m$ are pairwise distinct,
$y_i\in A_i$, and $x_{i\sigma}\in B_i$ ($1\leq i\leq m)$. Suppose that $g$
is an idempotent in $T(X)$ such that $e-g-f$. Then:
\begin{itemize}
  \item [\rm(1)] $x_jg=x_j$ and $y_jg=y_j$ for every $j\in\{1,\ldots,m\}$.
  \item [\rm(2)] If $1\leq i,j\leq m$ are such that
$A_i=\{x_i,y_i,z\}$, $B_j=\{y_j,x_{j\sigma},z\}$
and $A_i\cap B_j=\{z\}$, then $zg=z$.
\end{itemize}
\end{lemma}
\begin{proof}
Since $eg=ge$, $x_1g=x_i$ for some $i$. Then $x_ig=x_i$ (since $g$ is an idempotent). Thus, $e-g-f$ and
Lemma~\ref{lcen} imply that $y_ig=y_i$.
Since $x_i=x_{(i\sigma^{-1})\sigma}\in B_{i\sigma^{-1}}$ and $g$ commutes with $f$,
we have $y_{i\sigma^{-1}}g=y_{i\sigma^{-1}}$. But now, since $y_{i\sigma^{-1}}\in A_{i\sigma^{-1}}$ and $g$ commutes with $e$,
we have $x_{i\sigma^{-1}}g=x_{i\sigma^{-1}}$.
Continuing this way, we obtain $x_{i\sigma^{-k}}g=x_{i\sigma^{-k}}$ and $y_{i\sigma^{-k}}g=y_{i\sigma^{-k}}$
for every $k\in\{1,\ldots,m-1\}$. Since $\sigma$ is an $m$-cycle, it follows that $x_jg=x_j$ and $y_jg=y_j$
for every $j\in\{1,\ldots,m\}$. We have proved (1).

Suppose $A_i=\{x_i,y_i,z\}$, $B_j=\{y_j,x_{j\sigma},z\}$,
and $A_i\cap B_j=\{z\}$.
Then $zg\in\{x_i,y_i,z\}$ (since $x_ig=x_i$ and $eg=ge$) and $zg\in\{y_j,x_{j\sigma},z\}$ (since $y_jg=y_j$ and $fg=gf$).
Since $A_i\cap B_j=\{z\}$, we have $zg=z$, which proves (2).
\end{proof}

\begin{lemma}\label{lja3}
Let $n\geq4$. If $n\ne5$ or $r\ne4$, then for some idempotents $e,f\in J_r$, there is no idempotent $g\in J_r$
such that $e-g-f$.
\end{lemma}
\begin{proof}
Let $n\ne5$ or $r\ne4$.
Suppose that $r<n-1$ or $n$ is even. Then there is an integer $m$ such that $m\leq r$ and $r<2m\leq n$.
Let $e$ and $f$ be idempotents from Lemma~\ref{lja2}. Then $e,f\in J_r$ since $m\leq r$.
But every idempotent $g\in T(X)$ such that $e-g-f$ fixes at least $2m$ elements, and so $g\notin J_r$
since $r<2m$.

Suppose that $r=n-1$ and $n=2m+1$ is odd. Then $n\geq7$ since we are working under the assumption
that $n\ne5$ or $r\ne4$.
We again consider idempotents $e$ and $f$ from Lemma~\ref{lja2}, which belong to $J_r$
since $m<n-1=r$. Note that $X=\{x_1,\ldots,x_m,y_1,\ldots,y_m,z\}$.
We may assume that $z\in A_m$ and $z\in B_1$.
Since $n\geq7$, we have $m\geq3$. Thus, the intersection of $A_m=\{x_m,y_m,z\}$ and $B_1=\{y_1,x_2,z\}$
is $\{z\}$, and so $zg=z$ by Lemma~\ref{lja2}.
Hence $g=\id_{X}\notin J_r$, which concludes the proof.
\end{proof}

\begin{theorem}\label{tdia}
Let $n\geq3$ and let $J_r$ be an ideal in $T(X)$ such that $2\leq r<n$. Then:
\begin{itemize}
  \item [\rm(1)] If $n=3$ or $n=5$ and $r=4$, then the diameter of $\cge(J_r)$ is $2$.
  \item [\rm(2)] In all other cases, the diameter of $\cge(J_r)$ is $3$.
\end{itemize}
\end{theorem}
\begin{proof}
Suppose $n=3$ or $n=5$ and $r=4$. In these special cases, we obtained the desired result using GRAPE \cite{So06},
which is a package for GAP \cite{Scel92}.

Let $n\geq4$ and suppose that $n\ne5$ or $r\ne4$. By Lemmas~\ref{ljo1} and~\ref{lja1},
the diameter of $\cge(J_r)$ is at most $3$. By Lemma~\ref{lja3},
the diameter of $\cge(J_r)$ is at least $3$. Thus the diameter of $\cge(J_r)$
is $3$, which concludes the proof of~(2).
\end{proof}

\subsection{Commuting Graphs of Proper Ideals of $T(X)$}\label{ssprop}
In this subsection, we determine the diameter of every proper ideal of $T(X)$.
The ideal $J_1$ consists of the constant transformations,
so $\cg(J_1)$ is the graph with $n$ isolated vertices. Thus $J_1$ is not connected and its
diameter is $\infty$. Therefore, for the remainder of this subsection,
we assume that $n\geq3$ and $2\leq r<n$.

It follows from Lemma~\ref{lad1} and Theorem~\ref{tdia} that the diameter of $\cg(J_r)$ is at most $5$.
We will prove that this diameter is in fact $5$ except when $n=3$ or $n\in\{5,6,7\}$ and $r=4$.
It also follows from Lemma~\ref{lad1} that if $e$ and $f$ are idempotents in $J_r$, then
the distance between $e$ and $f$ in $\cg(J_r)$ is the same as the distance between $e$ and $f$
in $\cge(J_r)$. So no ambiguity will arise when we talk about the distance
between idempotents in $J_r$.

For $a\in T(X)$ and $x,y\in X$, we will write $x\ara y$ when $xa=y$.

\begin{lemma}\label{lad2}
Let $a,b\in T(X)$. Then $ab=ba$ if and only if for all $x,y\in X$, $x\ara y$ implies $xb\ara yb$.
\end{lemma}
\begin{proof}
Suppose $ab=ba$. Let $x,y\in X$ with $x\ara y$, that is, $y=xa$. Then, since $ab=ba$, we have
$yb=(xa)b=x(ab)=x(ba)=(xb)a$, and so $xb\ara yb$.

Conversely, suppose $x\ara y$ implies $xb\ara yb$ for all $x,y\in X$. Let $x\in X$.
Since $x\ara xa$, we have $xb\ara (xa)b$.
But this means that $(xb)a=(xa)b$, which implies $ab=ba$.
\end{proof}

Let $a\in T(X)$. Suppose $x_1,\ldots,x_m$ are pairwise distinct elements of $X$
such that $x_ia=x_{i+1}$ ($1\leq i<m$) and $x_ma=x_1$. We will then say that $a$
contains a cycle $(x_1\,x_2\ldots\,x_m)$.

\begin{lemma}\label{lad0}
Let $a\in J_r$ be a transformation containing a unique cycle $(x_1\,x_2\ldots\,x_m)$.
Let $e\in J_r$ be an idempotent such that $ae=ea$. Then
$x_ie=x_i$ for every $i\in\{1,\ldots,m\}$.
\end{lemma}
\begin{proof}
Since $a$ contains $(x_1\,x_2\ldots\,x_m)$, we have $x_1\ara x_2\ara\cdots\ara x_m\ara x_1$. Thus, by Lemma~\ref{lad2},
\[
x_1e\ara x_2e\ara\cdots\ara x_me\ara x_1e.
\]
Thus $(x_1e\,x_2e\ldots\,x_me)$ is a cycle in $a$, and is therefore equal to
$(x_1\,x_2\ldots\,x_m)$. Hence, for every $i\in\{1,\ldots,m\}$, there exists $j\in\{1,\ldots,m\}$
such that $x_i=x_je$,
and so $x_ie=(x_je)e=x_j(ee)=x_je=x_i$.
\end{proof}

To construct transformations
$a,b\in J_r$ such that the distance between $a$ and $b$ is $5$, it will be convenient
to introduce the following notation.

\begin{nota}\label{ndi5}
{\rm
Let $x_1,\ldots,x_m,z_1,\ldots,z_p$ be pairwise distinct elements of $X$, and let $s$ be fixed such that $1\leq s<p$.
We will denote by
\begin{equation}\label{e1ndi5}
a=(*\,z_s\ran(z_p\,z_{p-1}\ldots\,z_1\,x_1\ran(x_1\,x_2\ldots\,x_m)
\end{equation}
the transformation $a\in T(X)$ such that
\begin{align}
z_pa&=z_{p-1},\,z_{p-1}a=z_{p-2},\ldots,z_2a=z_1,\,z_1a=x_1,\notag\\
x_1a&=x_2,\,x_2a=x_3,\ldots,\,x_{m-1}a=x_m,\,x_ma=x_1,\notag
\end{align}
and $ya=z_s$ for all other $y\in X$.
Suppose $w\in X$
such that $w\notin\{x_1,\ldots,x_m,z_1,\ldots,z_p\}$ and $1\leq t<p$ with $t\ne s$. We will denote by
\begin{equation}\label{e2ndi5}
b=(*\,z_s\ran(w\,z_t\ran(z_p\,z_{p-1}\ldots\,z_1\,x_1\ran(x_1\,x_2\ldots\,x_m)
\end{equation}
the transformation $b\in T(X)$ that is defined as $a$ in (\ref{e1ndi5}) except that $wb=z_t$.
}
\end{nota}

\begin{lemma}\label{lad3}
Let $a\in J_r$ be the transformation defined in {\rm(\ref{e1ndi5})} such that  $m+p>r$.
Let $e\in J_r$ be an idempotent such that $ae=ea$. Then:
\begin{itemize}
  \item [\rm(1)] $x_ie=x_i$ for every $i\in\{1,\ldots,m\}$.
  \item [\rm(2)] $z_je=x_{m-j+1}$ for every $j\in\{1,\ldots,p\}$.
  \item [\rm(3)] $ye=x_{m-s}$ for every $y\in X-\{x_1,\ldots,x_m,z_1,\ldots,z_p\}$.
\end{itemize}
(We assume that for every integer $u$, $x_u=x_v$, where $v\in\{1,\ldots,m\}$ and $u\equiv v\pmod m$.)
\end{lemma}
\begin{proof}
Statement (1) follows from Lemma~\ref{lad0}.
By the definition of $a$, we have
\[
z_p\ara z_{p-1}\ara\cdots\ara z_1\ara x_1.
\]
Thus, by Lemma~\ref{lad2},
\[
z_pe\ara z_{p-1}e\ara\cdots\ara z_1e\ara x_1e=x_1.
\]
Since $z_1e\ara x_1$, either $z_1e=x_m$ or $z_1e\notin\{x_1,\ldots,x_m\}$.
We claim that the latter is impossible. Indeed, suppose $z_1e\notin\{x_1,\ldots,x_m\}$.
Then $z_je\notin\{x_1,\ldots,x_m\}$ for every $j\in\{1,\ldots,p\}$. Thus the set $\{x_1,\ldots,x_m,z_1e,\ldots,z_pe\}$
is a subset of $\ima(e)$ with $m+p$ elements. But this implies that $e\notin J_r$ (since $m+p>r$), which is a contradiction.
We proved the claim. Thus $z_1e=x_m$. Now, $z_2e\ara z_1e=x_m$, which implies $z_2e=x_{m-1}$. Continuing this way,
we obtain $z_3e=x_{m-2},\,z_4e=x_{m-3},\ldots$. (A special argument is required when $j=qm+1$ for some $q\geq1$.
Suppose $q=1$, that is, $j=m+1$.
Then $z_je\ara z_{j-1}e=z_me=x_1$, and so either $z_je=x_m$ or $z_je=z_1$.
But the latter is impossible
since we would have $x_m=z_1e=z_j(ee)=z_je=z_1$, which is a contradiction. Hence, for $j=m+1$,
we have $z_je=x_m$. Assuming, inductively, that $z_je=x_m$ for $j=qm+1$, we prove
by a similar argument that $z_je=x_m$ for $j=(q+1)m+1$.) This concludes the proof of (2).

Let $y\in X-\{x_1,\ldots,x_m,z_1,\ldots,z_p\}$.
Then $y\ara z_s$, and so $ye\ara z_se=x_{m-s+1}$. Suppose $s$ is not a multiple of $m$. Then $x_{m-s+1}\ne x_1$,
and so $ye\ara x_{m-s+1}$ implies $ye=x_{m-s}$. Suppose $s$ is a multiple of $m$. Then $ye\ara x_{m-s+1}=x_1$,
and so either $ye=x_m$ or $ye=z_1$.
But the latter is impossible
since we would have $x_m=z_1e=y(ee)=ye=z_1$, which is a contradiction. Hence, for $s$ that is a multiple of $m$,
we have $ye=x_m$, which concludes the proof of (3).
\end{proof}

The proof of the following lemma is almost identical to the proof of Lemma~\ref{lad3}.

\begin{lemma}\label{lad4}
Let $b\in J_r$ be the transformation defined in {\rm(\ref{e2ndi5})} such that $m+p>r$.
Let $e\in J_r$ be an idempotent such that $be=eb$. Then:
\begin{itemize}
  \item [\rm(1)] $x_ie=x_i$ for every $i\in\{1,\ldots,m\}$.
  \item [\rm(2)] $z_je=x_{m-j+1}$ for every $j\in\{1,\ldots,p\}$.
  \item [\rm(3)] $we=x_{m-t}$.
  \item [\rm(4)] $ye=x_{m-s}$ for every $y\in X-\{x_1,\ldots,x_m,z_1,\ldots,z_p,w\}$.
\end{itemize}
\end{lemma}

\begin{lemma}\label{lad6}
Let $n\in\{5,6,7\}$ and $r=4$. Then there are $a,b\in J_4$ such that the distance between $a$ and $b$ in $\cg(J_4)$ is at least $4$.
\end{lemma}
\begin{proof}
Let $a=(*\,4\ran(3\,4\,1\ran(1\,2)$ and $b=(*\,1\ran(2\,1\,3\ran(3\,4)$ (see Notation~\ref{ndi5}).
Suppose $e$ and $f$ are idempotents in $J_4$ such that $a-e$ and $f-b$. Then, by Lemma~\ref{lad3},
$e=(\{\ldots,3,1\},1\ran(\{4,2\},2\ran$ and $f=(\{\ldots,2,3\},3\ran(\{1,4\},4\ran$, where
``$\ldots$'' denotes ``$5$'' (if $n=5$), ``$5,6$'' (if $n=6$), and ``$5,6,7$'' (if $n=7$). Then
$e$ and $f$ do not commute, and so $d(e,f)\geq2$. Thus $d(a,b)\geq4$ by Lemma~\ref{lad1}.
\end{proof}

\begin{lemma}\label{lad7}
Let $n\in\{6,7\}$ and $r=4$. Let $a\in J_4$ be a transformation that is not an idempotent. Then there is an
idempotent $e\in J_4$ commuting with $a$ such that $\rank(e)\ne3$ or $\rank(e)=3$ and $ye^{-1}=\{y\}$
for some $y\in\ima(e)$.
\end{lemma}
\begin{proof}
If $a$ fixes some $x\in X$, then $a$ commutes with $e=(X,x\ran$ of rank $1$.
Suppose $a$ has no fixed points. Let $p$ be a positive integer such that $a^p$ is an idempotent.
If $a$ contains a unique cycle $(x_1\,x_2)$, then $e=a^p$ has rank $2$.
If $a$ contains a unique cycle $(x_1\,x_2\,x_3\,x_4)$
or two cycles $(x_1\,x_2)$ and $(y_1\,y_2)$ with $\{x_1,x_2\}\cap\{y_1,y_2\}=\emptyset$,
then $e=a^p$ has rank $4$.

Suppose $a$ contains a unique cycle $(x_1\,x_2\,x_3)$.
Define $e\in T(X)$ as follows. Set $x_ie=x_i$, $1\leq i\leq3$.

Suppose there are
$y,z\in X-\{x_1,x_2,x_3\}$ such that $ya=z$ and $za=x_i$ for some $i$.
We may assume that $za=x_1$.
Define $ze=x_3$ and $ye=x_2$. Let $u$ and $w$ be the two remaining elements in $X$
(only $u$ remains when $n=6$). Since $\rank(a)\leq4$, we have $\{u,w\}a\subseteq\{z,x_1,x_2,x_3\}$.
Suppose $ua=wa=z$. Define $ue=x_2$ and $we=x_2$.
Then $e$ is an idempotent of rank $3$ such that $ae=ea$ and $x_1e^{-1}=\{x_1\}$.
Suppose $ua$ or $wa$ is in $\{x_1,x_2,x_3\}$, say $ua\in\{x_1,x_2,x_3\}$.
Define $ue=u$, and $we=x_{i-1}$ (if $wa=x_i$), where $x_{i-1}=x_3$ if $i=1$, or $we=x_2$ (if $wa=z$).
Then $e$ is an idempotent of rank $4$ such that $ae=ea$.

Suppose that for every $y\in X-\{x_1,x_2,x_3\}$, $ya\in\{x_1,x_2,x_3\}$.
Select $z\in X-\{x_1,x_2,x_3\}$ and define $ze=z$. For every $y\in X-\{z,x_1,x_2,x_3\}$,
define $ye=x_{i-1}$ if $ya=x_i$.
Then $e$ is an idempotent of rank $4$ such that $ae=ea$.

Since $a\in J_4$, we have exhausted all possibilities, and the result follows.
\end{proof}

\begin{lemma}\label{lad8}
Let $n\in\{6,7\}$ and $r=4$. Then for all $a,b\in J_4$, the distance between $a$ and $b$ in $\cg(J_4)$ is at most $4$.
\end{lemma}
\begin{proof}
Let $a,b\in J_4$. If $a$ or $b$ is an idempotent, then $d(a,b)\leq4$ by Lemma~\ref{lad1} and Theorem~\ref{tdia}.
Suppose $a$ and $b$ are not idempotents. By Lemma~\ref{lad7}, there are idempotents $e,f\in J_4$
such that $ae=ea$, $bf=fb$, if $\rank(e)=3$, then
$ye^{-1}=\{y\}$ for some $y\in\ima(e)$, and if $\rank(f)=3$, then
$yf^{-1}=\{y\}$ for some $y\in\ima(f)$. We claim that there is an idempotent $g\in J_4$ such that $e-g-f$.
If $\ima(e)\cap\ima(f)\ne\emptyset$, then such an idempotent $g$ exists by Lemma~\ref{ljo1}.
Suppose $\ima(e)\cap\ima(f)=\emptyset$. Then, since $n\in\{6,7\}$, both $\rank(e)+\rank(f)\leq7$.
We may assume that $\rank(e)\leq\rank(f)$. There are six possible cases.
\vskip 1mm
\noindent{\bf Case 1.} $\rank(e)=1$.
\vskip 1mm
Then $e=(X,x\ran$ for some $x\in X$. Let $y=xf$.
Then $(x,y)\in\ima(e)\times\ima(f)$ and $(x,y)\in\Ker(e)\cap\Ker(f)$.
Thus, by Lemma~\ref{ljo2}, there is an idempotent $g\in J_4$ such that $e-g-f$.
\vskip 1mm
\noindent{\bf Case 2.} $\rank(e)=2$ and $\rank(f)=2$.
\vskip 1mm
We may assume that $e=(A_1,1\ran(A_2,2\ran$ and $f=(B_1,3\ran(B_2,4\ran$.
If $\{1,2\}\subseteq B_i$ or $\{3,4\}\subseteq A_i$ for some $i$, then
we can find $(x,y)\in\ima(e)\times\ima(f)$
such that $(x,y)\in\Ker(e)\cap\Ker(f)$, and so a desired idempotent $g$ exists by Lemma~\ref{ljo2}.
Otherwise, we may assume that $3\in A_1$ and $4\in A_2$. If $1\in B_1$ or $2\in B_2$, then
Lemma~\ref{ljo2} can be applied again. So suppose $1\in B_2$ and $2\in B_1$.
Now we have
\[
e=(\{\ldots,3,1\},1\ran(\{\ldots,4,2\},2\ran\mbox{ and }
f=(\{\ldots,2,3\},3\ran(\{\ldots,1,4\},4\ran.
\]
We define $g\in T(X)$ as follows. Set $xg=x$ for every $x\in\{1,2,3,4\}$. Let $x\in\{5,6,7\}$ ($x\in\{5,6\}$ if $n=6$).
If $x\in A_1\cap B_1$, define $xg=3$; if $x\in A_1\cap B_2$, define $xg=1$;
if $x\in A_2\cap B_1$, define $xg=2$; finally, if $x\in A_2\cap B_2$, define $xg=4$.
Then $g$ is an idempotent of rank $4$ and $e-g-f$.
\vskip 1mm
\noindent{\bf Case 3.} $\rank(e)=2$ and $\rank(f)=3$.
\vskip 1mm
We may assume that $e=(A_1,1\ran(A_2,2\ran$ and $f=(B_1,3\ran(B_2,4\ran(B_3,5\ran$.
If $\{3,4,5\}\subseteq A_1$ or $\{3,4,5\}\subseteq A_2$, then Lemma~\ref{ljo2} applies.
Otherwise, we may assume that $3,4\in A_1$ and $5\in A_2$. If $1\in B_1\cup B_2$ or $2\in B_3$,
then Lemma~\ref{ljo2} applies again.
So suppose $1\in B_3$ and $2\in B_1\cup B_2$. We may assume that $2\in B_1$.
Note that if $z\in\{6,7\}$, then $z$ cannot be in $B_2$ since $z\in B_2$ would imply that there is no
$y\in\ima(f)$ such that $yf^{-1}=\{y\}$.
So now
\[
e=(\{\ldots,3,4,1\},1\ran(\{\ldots,5,2\},2\ran\mbox{ and }
f=(\{\ldots,2,3\},3\ran(\{4\},4\ran(\{\ldots,1,5\},5\ran.
\]
We define $g\in T(X)$ as follows. Set $xg=x$ for every $x\in\{1,2,3,5\}$ and $4g=3$.
Let $z\in\{6,7\}$.
If $z\in A_1\cap B_1$, define $zg=3$; if $z\in A_1\cap B_3$, define $zg=1$;
if $z\in A_2\cap B_1$, define $zg=2$; finally, if $z\in A_2\cap B_3$, define $zg=5$.
Then $g$ is an idempotent of rank $4$ and $e-g-f$.
\vskip 1mm
\noindent{\bf Case 4.} $\rank(e)=2$ and $\rank(f)=4$.
\vskip 1mm
We may assume that $e=(A_1,1\ran(A_2,2\ran$ and $f=(B_1,3\ran(B_2,4\ran(B_3,5\ran(B_4,6\ran$.
If $\{3,4,5,6\}\subseteq A_1$ or $\{3,4,5,6\}\subseteq A_2$, then Lemma~\ref{ljo2} applies.
Otherwise, we may assume that $3,4,5\in A_1$ and $6\in A_2$ or $3,4\in A_1$ and $5,6\in A_2$.

Suppose $3,4,5\in A_1$ and $6\in A_2$. If $1\in B_1\cup B_2\cup B_3$ or $2\in B_4$, then Lemma~\ref{ljo2} applies.
So suppose $1\in B_4$, and we may assume that $2\in B_1$.
Now we have
\begin{align}
e&=(\{\ldots,3,4,5,1\},1\ran(\{\ldots,6,2\},2\ran,\notag\\
f&=(\{\ldots,2,3\},3\ran(\{\ldots,4\},4\ran(\{\ldots,5\},5\ran(\{\ldots,1,6\},6\ran.\notag
\end{align}
We define $g\in T(X)$ as follows. Set $xg=x$ for every $x\in\{1,2,3,6\}$, $4g=3$, and $5g=3$.
Define $7g=3$ if $7\in A_1$ and $7\in B_1\cup B_2\cup B_3$; $7g=1$ if $7\in A_1$ and $7\in B_4$;
$7g=2$ if $7\in A_2$ and $7\in B_1\cup B_2\cup B_3$; and $7g=6$ if $7\in A_2$ and $7\in B_4$.
Then $g$ is an idempotent of rank $4$ and $e-g-f$.
The argument in the case when $3,4\in A_1$ and $5,6\in A_2$ is similar.
\vskip 1mm
\noindent{\bf Case 5.} $\rank(e)=3$ and $\rank(f)=3$.
\vskip 1mm
Since both $e$ and $f$ have an element in their range whose preimage is the singleton,
we may assume that $e=(A_1,1\ran(A_2,2\ran(\{3\},3\ran$ and $f=(B_1,4\ran(B_2,5\ran(\{6\},6\ran$.
If $\{1,2\}\subseteq B_i$ or $\{4,5\}\subseteq A_i$ for some $i$, then
Lemma~\ref{ljo2} applies.
Otherwise, we may assume that $4\in A_1$ and $5\in A_2$. If $1\in B_1$ or $2\in B_2$, then
Lemma~\ref{ljo2} applies again. So suppose $1\in B_2$ and $2\in B_1$.
So now
\[
e=(\{\ldots,4,1\},1\ran(\{\ldots,5,2\},2\ran(\{3\},3\ran\mbox{ and }
f=(\{\ldots,2,4\},4\ran(\{\ldots,1,5\},5\ran(\{6\},6\ran.
\]
We define $g\in T(X)$ as follows. Set $xg=x$ for every $x\in\{1,2,4,5\}$, $3g=1$, and $6g=4$.
Define $7g=4$ if $7\in A_1$ and $7\in B_1$; $7g=1$ if $7\in A_1$ and $7\in B_2$;
$7g=2$ if $7\in A_2$ and $7\in B_1$; and $7g=5$ if $7\in A_2$ and $7\in B_2$.
Then $g$ is an idempotent of rank $4$ and $e-g-f$.
\vskip 1mm
\noindent{\bf Case 6.} $\rank(e)=3$ and $\rank(f)=4$.
\vskip 1mm
We may assume that $e=(A_1,1\ran(A_2,2\ran(\{3\},3\ran$ and $f=(B_1,4\ran(B_2,5\ran(B_3,6\ran(\{7\},7\ran$.
If $\{4,5,6\}\subseteq A_1$ or $\{4,5,6\}\subseteq A_2$, then Lemma~\ref{ljo2} applies.
So we may assume that $4,5\in A_1$ and $6\in A_2$. If $1\in B_1\cup B_2$ or $2\in B_3$, then
Lemma~\ref{ljo2} applies again. So we may assume that $1\in B_3$ and $2\in B_1$. So now
\begin{align}
e&=(\{\ldots,4,5,1\},1\ran(\{\ldots,6,2\},2\ran(\{3\},3\ran,\notag\\
f&=(\{\ldots,2,4\},4\ran(\{\ldots,5\},5\ran(\{\ldots,1,6\},6\ran(\{7\},7\ran.\notag
\end{align}
We define $g\in T(X)$ as follows. Set $xg=x$ for every $x\in\{1,2,4,6\}$ and $5g=4$.
Define $7g=4$ if $7\in A_1$; $7g=6$ if $7\in A_2$;
$3g=3$ if $3\in B_1\cup A_2$; and $3g=1$ if $3\in B_3$.
Then $g$ is an idempotent of rank $4$ and $e-g-f$.

\end{proof}

\begin{theorem}\label{tdia2}
Let $n\geq3$ and let $J_r$ be an ideal in $T(X)$ such that $2\leq r<n$. Then:
\begin{itemize}
   \item [\rm(1)] If $n=3$ or $n\in\{5,6,7\}$ and $r=4$, then the diameter of $\cg(J_r)$ is $4$.
   \item [\rm(2)] In all other cases, the diameter of $\cg(J_r)$ is $5$.
\end{itemize}
\end{theorem}
\begin{proof}
Let $n=3$. Then the diameter of $\cg(J_2)$ is at most $4$ by Lemma~\ref{lad1} and Theorem~\ref{tdia}.
On the other hand, consider $a=(3\,1\ran(1\,2)$ and $b=(2\,1\ran(1\,3)$ in $J_2$.
Suppose $e$ and $f$ are idempotents in $J_2$ such that $a-e$ and $f-b$. By Lemma~\ref{lad3},
$e=(\{1\},1\ran(\{3,2\},2\ran$ and $f=(\{1\},1\ran(\{2,3\},3\ran$.
Then $e$ and $f$ do not commute, and so $d(e,f)\geq2$. Thus $d(a,b)\geq4$ by Lemma~\ref{lad1},
and so the diameter of $\cg(J_2)$ is at least $4$.

Let $n\in\{5,6,7\}$ and $r=4$.
If $n=5$, then the diameter of $\cg(J_4)$ is at least $4$ (by Lemma~\ref{lad6})
and at most $4$ (by Lemma~\ref{lad1} and Theorem~\ref{tdia}).
If $n\in\{6,7\}$, then the diameter of $\cg(J_4)$ is at least $4$ (by Lemma~\ref{lad6})
and at most $4$ (by Lemma~\ref{lad8}). We have proved (1).

Let $n\geq4$ and suppose that $n\notin\{5,6,7\}$ or $r\ne4$. Then the diameter of $\cg(J_r)$ is at most $5$
by Lemma~\ref{lad1} and Theorem~\ref{tdia}. It remains to find $a,b\in J_r$ such that the distance
between $a$ and $b$ in $\cg(J_r)$ is at least $5$. We consider four possible cases.
\vskip 1mm
\noindent{\bf Case 1.} $r=2m-1$ for some $m\geq2$.
\vskip 1mm
Then $2\leq m<r<2m\leq n$. Let $x_1,\ldots,x_m,y_1,\ldots,y_m$ be pairwise distinct elements of $X$.
Let
\[
a=(*\,y_2\ran(y_1\,y_2\ldots\,y_m\,x_1\ran(x_1\,x_2\ldots\,x_m)\mbox{ and }
b=(*\,x_3\ran(x_2\,x_3\ldots\,x_{m-1}\,x_1\,y_1\ran(y_1\,y_2\ldots\,y_m)
\]
(see Notation~\ref{ndi5}) and note that $a,b\in J_r$ and $ab\ne ba$. Then, by Lemma~\ref{lad1},
there are idempotents $e_1,\ldots,e_k\in J_r$ ($k\geq1$) such that $a-e_1-\cdots-e_k-b$
is a minimal path in $\cg(J_r)$ from $a$ to $b$. By Lemma~\ref{lad3},
\[
e_1=(A_1,x_1\ran(A_2,x_2\ran\ldots(A_m,x_m\ran\mbox{ and }e_k=(B_1,y_1\ran(B_2,y_2\ran\ldots(B_m,y_m\ran,
\]
where $y_i\in A_i$ ($1\leq i\leq m$), $x_{i+1}\in B_i$ ($1\leq i<m)$, and $x_1\in B_m$.
Let $g\in T(X)$ be an idempotent such that $e_1-g-e_k$.
By Lemma~\ref{lja2}, $x_jg=x_j$ and $y_jg=y_j$ for every $j\in\{1,\ldots,m\}$. Hence
$\rank(g)\geq 2m>r$, and so $g\notin J_r$. It follows that the distance between $e_1$ and $e_k$
is at least $3$, and so the distance between $a$ and $b$ is at least $5$.
\vskip 1mm
\noindent{\bf Case 2.} $r=2m$ for some $m\geq3$.
\vskip 1mm
Then $3\leq m<r=2m<n$. Let $x_1,\ldots,x_m,y_1,\ldots,y_m,z$ be pairwise distinct elements of $X$.
Let
\begin{align}
a&=(*\,y_2\ran(z\,y_1\,y_2\ldots\,y_m\,x_1\ran(x_1\,x_2\ldots\,x_m),\notag\\
b&=(*\,x_1\ran(z\,x_3\ran(x_2\,x_3\ldots\,x_m\,x_1\,y_1\ran(y_1\,y_2\ldots\,y_m)\notag
\end{align}
(see Notation~\ref{ndi5}) and note that $a,b\in J_r$ and $ab\ne ba$. Then, by Lemma~\ref{lad1},
there are idempotents $e_1,\ldots,e_k\in J_r$ ($k\geq1$) such that $a-e_1-\cdots-e_k-b$
is a minimal path in $\cg(J_r)$ from $a$ to $b$. By Lemma~\ref{lad3},
\[
e_1=(A_1,x_1\ran(A_2,x_2\ran\ldots(A_m,x_m\ran\mbox{ and }e_k=(B_1,y_1\ran(B_2,y_2\ran\ldots(B_m,y_m\ran,
\]
where $y_i\in A_i$ ($1\leq i\leq m$), $x_{i+1}\in B_i$ ($1\leq i<m)$, $x_1\in B_m$,
$A_m=\{x_m,y_m,z\}$, and $B_1=\{y_1,x_2,z\}$. Let $g\in T(X)$ be an idempotent such that $e_1-g-e_k$.
By Lemma~\ref{lja2}, $x_jg=x_j$ and $y_jg=y_j$ for every $j\in\{1,\ldots,m\}$, and $zg=z$. Hence
$\rank(g)\geq 2m+1>r$, and so $g\notin J_r$. It follows that the distance between $e_1$ and $e_k$
is at least $3$, and so the distance between $a$ and $b$ is at least $5$.
\vskip 1mm
\noindent{\bf Case 3.} $r=4$.
\vskip 1mm
Since we are working under the assumption that $n\notin\{5,6,7\}$ or $r\ne4$, we have $n\notin\{5,6,7\}$.
Thus $n\geq8$ (since $r\leq n-1$). Let
\[
a=\begin{pmatrix}
1&2&3&4&5&6&7&8&9&\!\!\!\!\ldots\, n\\2&3&4&1&2&3&4&1&1&\!\!\!\!\ldots\, 1
\end{pmatrix}\mbox{ and }
b=\begin{pmatrix}
1&2&3&4&5&6&7&8&9&\!\!\!\!\ldots\, n\\5&6&7&8&6&7&8&5&1&\!\!\!\!\ldots\, 1
\end{pmatrix}.
\]
Note that $a,b\in J_4$, $ab\ne ba$, $(1\,2\,3\,4)$ is a unique cycle in $a$,
and $(5\,6\,7\,8)$ is a unique cycle in $b$. By Lemma~\ref{lad1},
there are idempotents $e_1,\ldots,e_k\in J_4$ ($k\geq1$) such that $a-e_1-\cdots-e_k-b$
is a minimal path in $\cg(J_4)$ from $a$ to $b$. By Lemma~\ref{lad0},
$ie_1=i$ and $(4+i)e_k=4+i$ for every $i\in\{1,2,3,4\}$.
By Lemma~\ref{lad2}, $5e_1=1$ or $5e_1=5$. But the latter is impossible since
with $5e_1=5$ we would have $\rank(e_1)\geq5$. Similarly, we obtain
$6e_1=2$, $7e_1=3$, $8e_1=4$, $2e_k=5$, $3e_k=6$, $4e_k=7$, and $1e_k=8$.
Let $g\in T(X)$ be an idempotent such that $e_1-g-e_k$.
By Lemma~\ref{lja2}, $jg=j$ for every $j\in\{1,\ldots,8\}$. Hence
$\rank(g)\geq 8>r$, and so $g\notin J_4$. It follows that the distance between $e_1$ and $e_k$
is at least $3$, and so the distance between $a$ and $b$ is at least~$5$.
\vskip 1mm
\noindent{\bf Case 4.} $r=2$.
\vskip 1mm
In this case we let
\[
a=\begin{pmatrix}
1&2&3&4&5&\!\!\!\!\ldots\, n\\2&1&2&1&1&\!\!\!\!\ldots\, 1
\end{pmatrix}\mbox{ and }
b=\begin{pmatrix}
1&2&3&4&5&\!\!\!\!\ldots\, n\\3&4&4&3&3&\!\!\!\!\ldots\, 3
\end{pmatrix}.
\]
Note that $a,b\in J_2$, $ab\ne ba$, $(1\,2)$ is a unique cycle in $a$,
and $(3\,4)$ is a unique cycle in $b$. By Lemma~\ref{lad1},
there are idempotents $e_1,\ldots,e_k\in J_2$ ($k\geq1$) such that $a-e_1-\cdots-e_k-b$
is a minimal path in $\cg(J_2)$ from $a$ to $b$. By Lemma~\ref{lad0},
$1e_1=1$, $2e_1=2$, $3e_k=3$, and $4e_k=4$.
By Lemma~\ref{lad2}, $3e_1=1$ or $3e_1=3$. But the latter is impossible since
with $3e_1=3$ we would have $\rank(e_1)\geq3$. Again By Lemma~\ref{lad2},
$4e_1=2$ or $4e_1=y$ for some $y\in\{4,5,\ldots,n\}$. But the latter is impossible
since we would have $ye_1=y$ and again $\rank(e_1)$ would be at least $3$.
Similarly, we obtain $2e_k=3$, and $1e_k=4$.
Let $g\in T(X)$ be an idempotent such that $e_1-g-e_k$.
By Lemma~\ref{lja2}, $jg=j$ for every $j\in\{1,\ldots,4\}$. Hence
$\rank(g)\geq 4>r$, and so $g\notin J_2$. It follows that the distance between $e_1$ and $e_k$
is at least $3$, and so the distance between $a$ and $b$ is at least~$5$.

Thus the diameter of $\cg(J_r)$ is at least $5$, which concludes the proof of (2).
\end{proof}

\subsection{The Commuting Graph of $T(X)$}\label{ssctx}
Let $X$ be a finite set with $|X|=n$. It has been proved in \cite[Theorem~3.1]{IrJa08} that if $n$ and $n-1$ are not prime,
then the diameter of the commuting graph of $\sym(X)$ is at most $5$, and that the bound is sharp since
the diameter of $\cg(\sym(X))$ is $5$ when $n=9$.
In this subsection, we determine the exact value of the diameter of the commuting graph of $T(X)$
for every $n\geq2$.

Throughout this subsection, we assume that $X$ is a finite set with $n\geq2$ elements.

\begin{lemma}\label{ltx}
Let $n\geq4$ be composite.
Let $a,f\in T(X)$ such that $a,f\ne\id_X$, $a\in\sym(X)$, and $f$ is an idempotent. Then
$d(a,f)\leq4$.
\end{lemma}
\begin{proof}
Fix $x\in\ima(f)$ and a cycle $(x_1\ldots x_m)$ of $a$ such that $x\in\{x_1,\ldots,x_m\}$.
Consider three cases.
\vskip 1mm
\noindent{\bf Case 1.} $a$ has a cycle $(y_1\ldots y_k)$ such that $k$ does not divide $m$.
\vskip 1mm
Then $a^m$ is different from $\id_X$ and it fixes $x$. Thus $a-a^m-(X,x\ran-f$, and so $d(a,f)\leq3$.
\vskip 1mm
\noindent{\bf Case 2.} $a$ has at least two cycles and for every cycle $(y_1\ldots y_k)$ of $a$, $k$ divides $m$.
\vskip 1mm
Suppose there is $z\in\ima(f)$ such that $z\in\{y_1,\ldots,y_k\}$ for some cycle $(y_1\ldots y_k)$
of $a$ different from $(x_1\ldots x_m)$. Since $k$ divides $m$, there is a positive integer $t$ such that $m=tk$.
Define $e\in T(X)$ by:
\begin{equation}\label{eltx1}
x_1e=y_1,\ldots,x_ke=y_k,\,x_{k+1}e=y_1,\ldots,x_{2k}e=y_k,\ldots,x_{(t-1)k+1}e=y_1,\ldots,x_{tk}e=y_k,
\end{equation}
and $ye=y$ for all other $y\in X$.
Then $e$ is an idempotent such that $ae=ea$ and $z\in\ima(e)$. Thus, by Lemma~\ref{ljo1},
$a-e-(X,z\ran-f$, and so $d(a,f)\leq3$.

Suppose that $\ima(f)\subseteq\{x_1,\ldots,x_m\}$. Consider any cycle $(y_1\ldots y_k)$ of $a$
different from $(x_1\ldots x_m)$. Since $\ima(f)\subseteq\{x_1,\ldots,x_m\}$, $y_1f=x_i$ for some $i$.
We may assume that $y_1f=x_1$.
Define an idempotent $e$ exactly as in (\ref{eltx1}).
Then $\ima(e)\cap\ima(f)=\emptyset$,
$(y_1,x_1)\in\ima(e)\times\ima(f)$, and $(y_1,x_1)\in\Ker(e)\cap\Ker(f)$.
Thus,  by Lemma~\ref{ljo2}, there is an idempotent $g\in T(X)-\{\id_X\}$ such that $e-g-f$.
Hence $a-e-g-f$, and so $d(a,f)\leq3$.
\vskip 1mm
\noindent{\bf Case 3.} $a$ is an $n$-cycle.
\vskip 1mm
Since $n$ is composite, there is a divisor $k$ of $n$ such that $1<k<n$. Then $a^k\ne\id_X$ is a permutation
with $k\geq2$ cycles, each of length $m=n/k$. By Case~2, $d(a^k,f)\leq3$, and so $d(a,f)\leq4$.
\end{proof}

\begin{lemma}\label{ltx1}
Let $n\geq4$ be composite.
Let $a,b\in T(X)$ such that $a,b\ne\id_X$ and $a\in\sym(X)$. Then $d(a,b)\leq5$.
\end{lemma}
\begin{proof}
Suppose $b\notin\sym(X)$. Then $b^k$ is an idempotent different from $\id_X$ for some $k\geq1$.
By Lemma~\ref{ltx}, $d(a,b^k)\leq4$, and so $d(a,b)\leq5$.

Suppose $b\in\sym(X)$. Suppose $n-1$ is not prime. Then, by \cite[Theorem~3.1]{IrJa08},
there is a path from $a$ to $b$
in $\cg(\sym(X))$ of length at most $5$. Such a path is also a path in $\cg(T(X))$, and so $d(a,b)\leq5$.
Suppose $p=n-1$ is prime. Then the proof of \cite[Theorem~3.1]{IrJa08} still works for $a$ and $b$ unless $a^p=\id_X$ or $b^p=\id_X$.
(See also \cite[Lemma~3.3]{IrJa08} and its proof.) Thus, if $a^p\ne\id_X$ and $b^p\ne\id_X$, then
there is a path from $a$ to $b$
in $\cg(\sym(X))$ of length at most $5$, and so $d(a,b)\leq5$. Suppose $a^p=\id_X$ or $b^p=\id_X$.
We may assume that $b^p=\id_X$. Then $b$ is a cycle of length $p$, that is, $b=(x_1\ldots x_p)(x)$.
Thus $b$ commutes with the constant idempotent $f=(X,x\ran$. By Lemma~\ref{ltx}, $d(a,f)\leq4$,
and so $d(a,b)\leq5$.
\end{proof}

\begin{lemma}\label{ltx2} Let $X=\{x_1,\ldots,x_m,y_1,\ldots,y_k\}$, $a\in\sym(X)$, and
$b=(y_1\ldots y_k\,x_1\ran(x_1\ldots x_m)$. If $ab=ba$ then $a=\id_X$.
\end{lemma}
\begin{proof}
Suppose $ab=ba$. By Lemma~\ref{lad2},
\begin{equation}\label{e1ltx2}
x_1a\arb x_2a\arb\cdots\arb x_ma\arb x_1a\quad\mbox{and}\quad y_1a\arb y_2a\arb\cdots\arb y_ka\arb x_1a.
\end{equation}
Since $(x_1\,x_2\ldots\,x_m)$ is a unique cycle in $b$, (\ref{e1ltx2}) implies that
\begin{equation}\label{e2ltx2}
x_1a=x_q,\, x_2a=x_{q+1},\ldots,\, x_ma=x_{q+m-1},
\end{equation}
where $q\in\{1,\ldots,m\}$ ($x_{q+i}=x_{q+i-m}$ if $q+i>m$).
Thus $x_1a=x_j$ for some $j$. Since $y_k\arb x_1$ and $x_m\arb x_1$,
we have $y_ka\arb x_1a=x_j$ and $x_ma\arb x_1a=x_j$.  Suppose $j\geq2$. Then $x_jb^{-1}=\{x_{j-1}\}$,
and so $y_ka=x_{j-1}=x_ma$. But this implies $y_k=x_m$ (since $a$ is injective), which is a contradiction.
Hence $j=1$, and so $x_1a=x_1$. But then $x_ia=x_i$ for all $i$ by (\ref{e2ltx2}).

Since $y_ka\arb x_1a=x_1$, we have $y_ka=y_k$ since $x_1b^{-1}=\{y_k,x_m\}$. Let $i\in\{1,\ldots,k-1\}$
and suppose $y_{i+1}a=y_{i+1}$. Then $y_ia=y_i$
since $y_ia\arb y_{i+1}a=y_{i+1}$ and $y_{i+1}b^{-1}=\{y_{i+1}\}$. It follows that $y_ia=y_i$ for all $i\in\{1,\ldots,k\}$.
\end{proof}

\begin{lemma}\label{ltx3}
Let $m$ be a positive integer such that $2m\leq n$, $\sigma$ be an $m$-cycle on $\{1,\ldots,m\}$, $a\in\sym(X)$, and
\[
e=(A_1,x_1\ran(A_2,x_2\ran\ldots(A_m,x_m\ran\mbox{ and }f=(B_1,y_1\ran(B_2,y_2\ran\ldots(B_m,y_m\ran
\]
be idempotents in $T(X)$ such that $x_1,\ldots,x_m,y_1,\ldots,y_m$ are pairwise distinct,
$y_i\in A_i$, and $x_{i\sigma}\in B_i$ ($1\leq i\leq m)$. Then:
\begin{itemize}
  \item [\rm(1)] Suppose $X=\{x_1,\ldots,x_m,y_1,\ldots,y_m,z\}$ and $z\in A_i\cap B_j$ such that $A_i\cap B_j=\{z\}$.
If $e-a-f$, then $a=\id_X$.
  \item [\rm(2)] Suppose $X=\{x_1,\ldots,x_m,y_1,\ldots,y_m,z,w\}$, $z\in A_i\cap B_j$ such that $A_i\cap B_j=\{z\}$,
and $w\in A_s\cap B_t$ such that $A_s\cap B_t=\{w\}$, where $s\ne i$ and $t\ne j$.
If $e-a-f$, then $a=\id_X$.
\end{itemize}
\end{lemma}
\begin{proof}
To prove (1), suppose $e-a-f$ and note that $A_i=\{x_i,y_i,z\}$ and $B_j=\{y_j,x_{j\sigma},z\}$. By Lemma~\ref{lcen}, there is $p\in\{1,\ldots,m\}$
such that $x_ia=x_p$ and $A_ia\subseteq A_p$. Suppose $p\ne i$. Then $A_p=\{x_p,y_p\}$, and so
$A_ia$ cannot be a subset of $A_p$ since $a$ is injective. It follows that $p=i$, that is,
$x_ia=x_i$ and $A_ia\subseteq A_i$. Similarly, $y_ja=y_j$ and $B_ja\subseteq B_j$. Thus $za\in A_i\cap B_j=\{z\}$,
and so $za=z$. Hence, since $a$ is injective, $y_ia=y_i$.

We have proved that $x_ia=x_i$, $y_ia=y_i$, and $za=z$. We have $B_i=\{y_i,x_{i\sigma}\}$ or $B_i=\{y_i,x_{i\sigma},z\}$
Since $y_ia=y_i$, we have $B_ia\subseteq B_i$ by Lemma~\ref{lcen}.
Since $za=z$ and $a$ is injective, it follows that $x_{i\sigma}a=x_{i\sigma}$. By the foregoing argument
applied to $A_{i\sigma}=\{x_{i\sigma},y_{i\sigma}\}$, we obtain $y_{i\sigma}a=y_{i\sigma}$.
Continuing this way, we obtain $x_{i\sigma^k}a=x_{i\sigma^k}$ and $y_{i\sigma^k}a=y_{i\sigma^k}$
for every $k\in\{1,\ldots,m-1\}$. Since $\sigma$ is an $m$-cycle, it follows that $x_ja=x_j$ and $y_jg=y_j$
for every $j\in\{1,\ldots,m\}$. Hence $a=\id_X$. We have proved (1). The proof of (2) is similar.
\end{proof}

\begin{theorem}\label{tdia3}
Let $X$ be a finite set with $n\geq2$ elements. Then:
\begin{itemize}
   \item [\rm(1)] If $n$ is prime, then $\cg(T(X))$ is not connected.
   \item [\rm(2)] If $n=4$, then the diameter of $\cg(T(X))$ is $4$.
   \item [\rm(3)] If $n\geq6$ is composite, then the diameter of $\cg(T(X))$ is $5$.
\end{itemize}
\end{theorem}
\begin{proof}
Suppose $n=p$ is prime. Consider a $p$-cycle $a=(x_1\,x_2\ldots x_p)$ and let $b\in T(X)$ be such that $b\ne\id_X$
and $ab=ba$. Let $x_q=x_1b$. Then, by Lemma~\ref{lad2}, $x_ib=x_{q+i}$ for every $i\in\{1,\ldots,p\}$
(where $x_{q+i}=x_{q+i-m}$ if $q+i>m$). Thus $b=a^q$, and so, since $p$ is prime, $b$ is also a $p$-cycle. It follows that
if $c$ is a vertex of $\cg(T(X))$ that is not a $p$-cycle, then there is no path in $\cg(T(X))$ from
$a$ to $c$. Hence $\cg(T(X))$ is not connected. We have proved (1).

We checked the case $n=4$ directly using GRAPE \cite{So06}
through GAP \cite{Scel92}. We found that, when $|X|=4$, the diameter of $\cg(T(X))$ is $4$.

Suppose $n\geq6$ is composite.
Let $a,b\in T(X)$ such that $a,b\ne\id_X$. If $a\in\sym(X)$ or $b\in\sym(X)$, then $d(a,b)\leq5$
by Lemma~\ref{ltx1}. If $a,b\notin\sym(X)$, then $a,b\in J_{n-1}$, and so $d(a,b)\leq5$
by Theorem~\ref{tdia2}. Hence the diameter of $\cg(T(X))$ is at most $5$. It remains to find $a,b\in T(X)-\{\id_X\}$
such that $d(a,b)\geq5$.

For $n\in\{6,8\}$, we employed GAP \cite{Scel92}. When $n=6$, we found that the distance between the $6$-cycle
$a=(1\,2\,3\,4\,5\,6)$ and
$b=\begin{pmatrix}
1&2&3&4&5&6\\2&3&5&1&2&4\end{pmatrix}$ in $\cg(T(X))$ is at least $5$. And when $n=8$, the distance between
the $8$-cycle $a=(1\,2\,3\,4\,5\,6\,7\,8)$ and
$b=\begin{pmatrix}
1&2&3&4&5&6&7&8\\2&3&1&1&4&8&6&5\end{pmatrix}$ in $\cg(T(X))$ is at least $5$.

To verify this with GAP, we used the following sequence of arguments and computer calculations:
\begin{enumerate}
 \item By Lemma \ref{lad1}, if there exists a path  $a-c_1-c_2-\ldots-c_k-b$, then there exists a path  $a-e_1-e_2-\ldots-e_k-b$, where
each $e_i$ is either an idempotent or a permutation;
\item Let $E$ be the set idempotents of $T(X)-\{\id_X\}$ and let $G=\sym(X)-\{\id_X\}$.
For $A\subseteq T(X)$, let $C(A)=\{f\in E\cup G:(\exists_{a\in A}) af=fa\}$;
\item Calculate $C(C(\{a\}))$ and $C(\{b\})$;
\item Verify that for all $c\in C(C(\{a\}))$ and all $d\in C(\{b\})$, $cd\neq dc$;
\item If there were a path $a-c_1-c_2-c_3-b$ from $a$ to $b$, then we would have
$c_2\in C(C(\{a\}))$, $c_3 \in C(\{b\})$, and $c_2c_3=c_3c_2$. But, by 4., there are no such $c_2$ and $c_3$,
and it follows that the distance between $a$ and $b$ is at least $5$.
\end{enumerate}

Let $n\geq9$ be composite. We consider two cases.
\vskip 1mm
\noindent{\bf Case 1.} $n=2m+1$ is odd ($m\geq4$).
\vskip 1mm
Let $X=\{x_1,\ldots,x_m,y_1,\ldots,y_m,z\}$. Consider
\[
a=(z\,y_1\,y_2\ldots\,y_m\,x_1\ran(x_1\,x_2\ldots\,x_m)\mbox{ and }
b=(x_2\,x_3\ldots\,x_m\,x_1\,z\,y_2\ran(y_1\,y_2\ldots\,y_m).
\]
Let $\lam$ be a minimal path in $\cg(T(X))$ from $a$ to $b$.
By Lemma~\ref{ltx2}, there is no $g\in\sym(X)$ such that $g\ne\id_X$ and $ag=ga$ or $bg=gb$.
Thus, by the proof of Lemma~\ref{lad1}, $\lam=a-e_1-\cdots-e_k-b$, where $e_1$ and $e_k$ are idempotents.
By Lemma~\ref{lad3},
\[
e_1=(A_1,x_1\ran(A_2,x_2\ran\ldots(A_m,x_m\ran\mbox{ and }e_k=(B_1,y_1\ran(B_2,y_2\ran\ldots(B_m,y_m\ran,
\]
where $y_i\in A_i$ ($1\leq i\leq m$), $x_{i+1}\in B_i$ ($1\leq i<m)$, $x_1\in B_m$,
$A_m=\{x_m,y_m,z\}$, and $B_1=\{y_1,x_2,z\}$. Since $m\geq4$, $A_m\cap B_1=\{z\}$. Thus, by Lemma~\ref{ltx3},
there is no $g\in\sym(X)$ such that $g\ne\id_X$ and $e_1-g-e_k$.
Hence, if $\lam$ contains an element
$g\in\sym(X)$, then the length of $\lam$ is at least $5$. Suppose $\lam$ does not contain any permutations.
Then $\lam$ is a path in $J_{n-1}$ and we may assume that all vertices in $\lam$ except $a$ and $b$ are idempotents
(by Lemma~\ref{lad3}). By Lemma~\ref{lja2}, there is no idempotent $f\in J_{n-1}$ such that
$e_1-f-e_k$. (Here, the $m$-cycle that occurs in Lemmas~\ref{lja2} and~\ref{ltx3} is $\sigma=(1\,2\ldots m)$.)
Hence the length of $\lam$ is at least $5$.

\vskip 1mm
\noindent{\bf Case 2.} $n=2m+2$ is even ($m\geq4$).
\vskip 1mm
Let $X=\{x_1,\ldots,x_m,y_1,\ldots,y_m,z,w\}$. Consider
\[
a=(z\,y_1\,y_2\ldots\,y_m\,w\,x_2\ran(x_1\,x_2\ldots\,x_m)\mbox{ and }
b=(w\,x_2\,x_3\ldots\,x_{m-2}\,x_m\,x_1\,x_{m-1}\,y_2\ran(y_1\,y_2\ldots\,y_m).
\]
Let $\lam$ be a minimal path in $\cg(T(X))$ from $a$ to $b$.
By Lemma~\ref{ltx2}, there is no $g\in\sym(X)$ such that $g\ne\id_X$ and $ag=ga$ or $bg=gb$.
Thus, by the proof of Lemma~\ref{lad1}, $\lam=a-e_1-\cdots-e_k-b$, where $e_1$ and $e_k$ are idempotents.
By Lemma~\ref{lad3},
\[
e_1=(A_1,x_1\ran(A_2,x_2\ran\ldots(A_m,x_m\ran\mbox{ and }e_k=(B_1,y_1\ran(B_2,y_2\ran\ldots(B_m,y_m\ran,
\]
where $y_i\in A_i$ ($1\leq i\leq m$), $x_{i+1}\in B_i$ ($1\leq i\leq m-3)$, $x_m\in B_{m-2}$,
$x_1\in B_{m-1}$, $x_{m-1}\in B_m$, $A_1=\{x_1,y_1,w\}$,
$A_m=\{x_m,y_m,z\}$, $B_1=\{y_1,x_2,z\}$, and $B_m=\{y_m,x_{m-1},w\}$.
Since $m\geq4$, $A_m\cap B_1=\{z\}$ and $A_1\cap B_m=\{w\}$. Thus, by Lemma~\ref{ltx3},
there is no $g\in\sym(X)$ such that $g\ne\id_X$ and $e_1-g-e_k$. Hence, as in Case 1, the length of $\lam$
is at least $5$. (Here, the $m$-cycle that occurs in Lemmas~\ref{lja2} and~\ref{ltx3} is $\sigma=(1,2\ldots, m-3,m-2,m,m-1)$.)

Hence, if $n\geq6$ is composite, then the diameter of $\cg(T(X))$ is $5$. This concludes the proof.
\end{proof}

\section{Minimal Left Paths}\label{smlp}
\setcounter{equation}{0}

In this section, we prove that for every integer $n\geq4$, there is a band $S$ with knit degree $n$.
We will show how to construct such an $S$ as a subsemigroup of $T(X)$ for some finite set $X$.

Let $S$ be a finite non-commutative
semigroup. Recall that a path $a_1-a_2-\cdots-a_m$ in $\cg(S)$ is called a \emph{left path}
(or $l$-path) if $a_1\ne a_m$ and $a_1a_i=a_ma_i$ for every $i\in\{1,\ldots,m\}$. If there is any $l$-path in $\cg(S)$,
we define the \emph{knit degree} of $S$, denoted $\kd(S)$, to be the length of a shortest $l$-path in $\cg(S)$.
We say that an $l$-path $\lam$ from $a$ to $b$ in $\cg(S)$
is a \emph{minimal $l$-path} if there is no $l$-path from $a$ to $b$ that is shorter than $\lam$.

\subsection{The Even Case}\label{sseven}

In this subsection, we will construct a band of knit degree $n$ where $n\geq4$ is even.
For $x\in X$, we denote by $c_x$ the constant transformation with image $\{x\}$.
The following lemma is obvious.

\begin{lemma}\label{lcon}
Let $c_x,c_y,e\in T(X)$ such that $e$ is an idempotent. Then:
\begin{itemize}
  \item [\rm(1)] $c_xe=ec_x$ if and only if $x\in\ima(e)$.
  \item [\rm(2)] $c_xe=c_ye$ if and only if $(x,y)\in\Ker(e)$.
\end{itemize}
\end{lemma}

Now, given an even $n\geq4$, we will construct a band $S$ such that $\kd(S)=n$.
We will explain the construction using $n=8$ as an example.
The band $S$ will be a subsemigroups of $T(X)$, where
\[
X=\{y_0,y_1,y_2,y_3,y_4=v_0,v_1,v_2,v_3,v_4,x_1,x_2,x_3,x_4,u_1,u_2,u_3,u_4,r,s\},
\]
and it will be generated by idempotent transformations $a_1,a_2,a_3,a_4,b_1,b_2,b_3,b_4,e_1$,
whose images are defined by Table~1.

\begin{table}[h]
\[
\begin{tabular}{|c|ccc|}\hline
     $\ima(a_1)$ & $y_0$ &  $x_1$  & $y_1$ \\\hline
     $\ima(a_2)$ & $y_1$ &  $x_2$  & $y_2$ \\\hline
     $\ima(a_3)$ & $y_2$ &  $x_3$  & $y_3$ \\\hline
     $\ima(a_4)$ & $y_3$ &  $x_4$  & $y_4$ \\\hline
     $\ima(b_1)$ & $y_4$ &  $u_1$  & $v_1$ \\\hline
     $\ima(b_2)$ & $v_1$ &  $u_2$  & $v_2$ \\\hline
     $\ima(b_3)$ & $v_2$ &  $u_3$  & $v_3$ \\\hline
     $\ima(b_4)$ & $v_3$ &  $u_4$  & $v_4$ \\\hline
     $\ima(e_1)$ & $v_4$ &  $r$  & $s$ \\\hline
\end{tabular}
\]
\caption{Images of the generators.}
\end{table}
We will define the kernels in such a way that the generators with the same subscript
will have the same kernel. For example, $\Ker(a_1)=\Ker(b_1)=\Ker(e_1)$ and $\Ker(a_2)=\Ker(b_2)$.
Let $i\in\{2,3,4\}$. The kernel of $a_i$ will have the following three classes (elements of the partition $X/\Ker(a_i)$):
\begin{align}
\mbox{Class-1}&=\ima(a_{i+1})\cup\ldots\cup\ima(a_4)\cup\ima(b_1)\cup\ldots\cup\ima(b_{i-1}),\notag\\
\mbox{Class-2}&=\ima(b_{i+1})\cup\ldots\cup\ima(b_4)\cup\ima(e_1)\cup\ima(a_1)\cup\ldots\cup\ima(a_{i-1}),\notag\\
\mbox{Class-3}&=\{x_i,u_i\}.\notag
\end{align}
For example, $\Ker(a_2)$ has the following classes:
\begin{align}
\mbox{Class-1}&=\{y_2,x_3,y_3,x_4,y_4,u_1,v_1\},\notag\\
\mbox{Class-2}&=\{v_2,u_3,v_3,u_4,v_4,r,s,y_0,x_1,y_1\},\notag\\
\mbox{Class-3}&=\{x_2,u_2\}.\notag
\end{align}
We define the kernel of $a_1$ as follows:
\begin{align}
\mbox{Class-1}&=\ima(a_2)\cup\ima(a_3)\cup\ima(a_4)\cup\{s\}=\{y_1,x_2,y_2,x_3,y_3,x_4,y_4,s\},\notag\\
\mbox{Class-2}&=\ima(b_2)\cup\ima(b_3)\cup\ima(b_4)\cup\{y_0\}=\{v_1,u_2,v_2,u_3,v_3,u_4,v_4,y_0\},\notag\\
\mbox{Class-3}&=\{x_1,u_1,r\}.\notag
\end{align}

Now the generators are completely defined since $\Ker(b_i)=\Ker(a_i)$, $1\leq i\leq 4$, and $\Ker(e_1)=\Ker(a_1)$.
Order the generators as follows:
\begin{equation}\label{elist1}
a_1,\,a_2,\,a_3,\,a_4,\,b_1,\,b_2,\,b_3,\,b_4,\,e_1.
\end{equation}
Let $S$ be the semigroup generated by the idempotents listed in (\ref{elist1}).
Since the idempotents with the same subscript have the same kernel, they form a right-zero subsemigroup of $S$.
For example, $\{a_1,b_1,e_1\}$ is a right-zero semigroup: $a_1a_1=b_1a_1=e_1a_1=a_1$,
$a_1b_1=b_1b_1=e_1b_1=b_1$, and $a_1e_1=b_1e_1=e_1e_1=e_1$. The product of any two generators with different
indices is a constant transformation. For example, $a_2a_4=c_{y_3}$, $a_4a_2=c_{y_2}$,
and $a_1b_3=c_{v_3}$. The semigroup $S$ consists of the nine generators listed in (\ref{elist1}) and
$10$ constants:
\[
S=\{a_1,a_2,a_3,a_4,b_1,b_2,b_3,b_4,e_1,c_{y_0},c_{y_1},c_{y_2},c_{y_3},c_{y_4},c_{v_1},c_{v_2},c_{v_3},c_{v_4},c_s\},
\]
so $S$ is a band. Note that $Z(S)=\emptyset$.
Each idempotent in (\ref{elist1}) commutes with the next idempotent, so
$a_1-a_2-a_3-a_4-b_1-b_2-b_3-b_4-e_1$ is a path in $\cg(S)$. Moreover, it is a unique $l$-path in $\cg(S)$,
so $\kd(S)=8$.

We will now provide a general construction of a band $S$ such that $\kd(S)=n$, where $n$ is even.

\begin{defi}\label{dco1}
{\rm
Let $k\geq2$ be an integer. Let
\[
X=\{y_0,y_1,\ldots,y_k=v_0,v_1,\ldots,v_k,x_1,\ldots,x_k,u_1,\ldots,u_k,r,s\}.
\]
We will define idempotents $a_1,\ldots,a_k,b_1,\ldots,b_k,e_1$ as follows. For $i\in\{1,\ldots,k\}$, let
\begin{align}
\ima(a_i)&=\{y_{i-1},x_i,y_i\},\notag\\
\ima(b_i)&=\{v_{i-1},u_i,v_i\},\notag\\
\ima(e_1)&=\{v_k,r,s\}.\notag
\end{align}

For $i\in\{2,\ldots,k\}$, define the $\Ker(a_i)$-classes by:
\begin{align}
\mbox{Class-1}&=\ima(a_{i+1})\cup\ldots\cup\ima(a_k)\cup\ima(b_1)\cup\ldots\cup\ima(b_{i-1}),\notag\\
\mbox{Class-2}&=\ima(b_{i+1})\cup\ldots\cup\ima(b_k)\cup\ima(e_1)\cup\ima(a_1)\cup\ldots\cup\ima(a_{i-1}),\notag\\
\mbox{Class-3}&=\{x_i,u_i\}.\notag
\end{align}
(Note that for $i=k$, $\mbox{Class-1}=\ima(b_1)\cup\ldots\cup\ima(b_{k-1})$ and
$\mbox{Class-2}=\ima(e_1)\cup\ima(a_1)\cup\ldots\cup\ima(a_{i-1})$.)

Define the $\Ker(a_1)$-classes by:
\begin{align}
\mbox{Class-1}&=\ima(a_2)\cup\ldots\cup\ima(a_k)\cup\{s\},\notag\\
\mbox{Class-2}&=\ima(b_2)\cup\ldots\cup\ima(b_k)\cup\{y_0\},\notag\\
\mbox{Class-3}&=\{x_1,u_1,r\}.\notag
\end{align}
Let $\Ker(b_i)=\Ker(a_i)$ for every $i\in\{1,\ldots,k\}$, and $\Ker(e_1)=\Ker(a_1)$.
Now, define the subsemigroup $S_0^k$ of $T(X)$ by:
\begin{equation}\label{edco1}
S_0^k=\mbox{the semigroup generated by $\{a_1,\ldots,a_k,b_1,\ldots,b_k,e_1\}$.}
\end{equation}
}
\end{defi}

We must argue that the idempotents $a_1,\ldots,a_k,b_1,\ldots,b_k,e_1$ are well defined, that is,
for each of them, different elements of the image lie in different kernel classes.
Consider $a_i$, where $i\in\{2,\ldots,k\}$. Then $\ima(a_i)=\{y_{i-1},x_i,y_i\}$.
Then $y_i$ lies in Class-1 (see Definition~\ref{dco1}) since $y_i\in\ima(a_{i+1})$
(or $y_i\in\ima(b_1)$ if $i=k$), $y_{i-1}$ lies in Class-2 since $y_{i-1}\in\ima(a_{i-1})$,
and $x_i$ lies in Class-3. Arguments for the remaining idempotents are similar.

For the remainder of this subsection, $S_0^k$ will be the semigroup (\ref{edco1}).
Our objective is to prove
that $S_0^k$ is a band such that $\pi=a_1-\cdots-a_k-b_1-\cdots-b_k-e_1$ is a shortest $l$-path in $S_0^k$.
Since $\pi$ has length $2k=n$, it will follow that $S_0^k$ is a band with knit degree $n$.

We first analyze products of the generators of $S_0^k$.

\begin{lemma}\label{lev2}
Let $1\leq i<j\leq k$. Then:
\begin{itemize}
  \item [\rm(1)] $a_ib_i=b_i$, $b_ia_i=a_i$, $a_1e_1=b_1e_1=e_1$, $e_1a_1=b_1a_1=a_1$, and $e_1b_1=a_1b_1=b_1$.
  \item [\rm(2)] $a_ia_j=c_{y_{j-1}}$ and $a_ja_i=c_{y_i}$.
  \item [\rm(3)] $a_ib_j=c_{v_j}$ and $a_jb_i=c_{v_{i-1}}$.
  \item [\rm(4)] $b_ia_j=c_{y_j}$ and $b_ja_i=c_{y_{i-1}}$.
  \item [\rm(5)] $b_ib_j=c_{v_{j-1}}$ and $b_jb_i=c_{v_i}$.
  \item [\rm(6)] $e_1a_j=c_{y_{j-1}}$ and $a_je_1=c_s$.
  \item [\rm(7)] $e_1b_j=c_{v_j}$ and $b_je_1=c_{v_k}$.
\end{itemize}
\end{lemma}
\begin{proof}
Statement (1) is true because the generators of $S_0^k$ are idempotents and the ones with the same subscript
have the same kernel.
By Definition~\ref{dco1}, Class-2 of $\Ker(a_j)$ contains both $\ima(a_{j-1})=\{y_{j-2},x_{j-1},y_{j-1}\}$
and $\ima(a_i)$ (since $i<j$). Since $y_{j-1}\in\ima(a_j)=\{y_{j-1},x_j,y_j\}$, $a_j$ maps all elements
of Class-2 to $y_{j-1}$. Hence $a_ia_j=c_{y_{j-1}}$. Similarly, since $i<j$, Class-1 of $\Ker(a_i)$ contains
both $\ima(a_{i+1})=\{y_i,x_{i+1},y_{i+1}\}$
and $\ima(a_j)$. Since $y_i\in\ima(a_i)=\{y_{i-1},x_i,y_i\}$, $a_i$ maps all elements
of Class-1 to $y_i$. Hence $a_ja_i=c_{y_i}$. We have proved (2). Proofs of (3)-(7) are similar.
For example, $b_je_1=c_{v_k}$ because Class-2 of $\Ker(e_1)=\Ker(a_1)$ contains both $\ima(b_j)$
and $\ima(b_k)=\{v_{k-1},u_k,v_k\}$, and $v_k\in\ima(e_1)$.
\end{proof}

The following corollaries are immediate consequences of Lemma~\ref{lev2}.

\begin{corollary}\label{cev2}
The semigroup $S_0^k$ is a band. It consists of $2k+1$ generators from Definition~{\rm \ref{dco1}}
and $2k+2$ constant transformations:
\[
S_0^k=\{a_1,\ldots,a_k,b_1,\ldots,b_k,e_1,c_{y_0},c_{y_1},\ldots,c_{y_k},c_{v_1},\ldots,c_{v_k},c_s\}.
\]
\end{corollary}
\begin{corollary}\label{cev2a}
Let $g,h\in S_0^k$ be generators from the list
\begin{equation}\label{ecev2}
a_1,\ldots,a_k,b_1,\ldots,b_k,e_1.
\end{equation}
Then $gh=hg$ if and only if $g$ and $h$ are consecutive elements in the list.
\end{corollary}

Lemma~\ref{lev2} gives a partial multiplication table for $S_0^k$. The following lemma
completes the table.

\begin{lemma}\label{lev2a}
Let $1\leq p\leq k$ and $1\leq i<j\leq k$. Then:
\begin{itemize}
  \item [\rm(1)] $c_{y_p}a_p=c_{y_p}$,
$c_{y_p}b_p=c_{v_{p-1}}$, $c_{y_i}a_j=c_{y_{j-1}}$, $c_{y_j}a_i=c_{y_i}$,
$c_{y_i}b_j=c_{v_j}$, $c_{y_j}b_i=c_{v_{i-1}}$, $c_{y_p}e_1=c_s$,
$c_{y_0}a_p=c_{y_{p-1}}$, $c_{y_0}b_p=c_{v_p}$, and $c_{y_0}e_1=c_{v_k}$.
  \item [\rm(2)] $c_{v_p}a_p=c_{y_{p-1}}$, $c_{v_p}b_p=c_{v_p}$,
$c_{v_i}a_j=c_{y_j}$, $c_{v_j}a_i=c_{y_{i-1}}$, $c_{v_i}b_j=c_{v_{j-1}}$, $c_{v_j}b_i=c_{v_i}$, and $c_{v_p}e_1=c_{v_k}$.
  \item [\rm(3)] $c_sa_j=c_{y_{j-1}}$, $c_sb_j=c_{v_j}$, $c_sa_1=c_{y_1}$, $c_sb_1=c_{v_0}$, and $c_se_1=c_s$.
\end{itemize}
\end{lemma}
\begin{proof}
We have
$c_{y_p}a_p=c_{y_p}$ since $y_p\in\ima(a_p)$.
By Definition~\ref{dco1}, Class-1 of $\Ker(b_p)$ contains both $\ima(a_{p+1})$
and $\ima(b_{p-1})$. Since $y_p\in\ima(a_{p+1})$ and $v_{p-1}\in\ima(b_{p-1})$, both $y_p$ and $v_{p-1}$
are in Class-1. Hence $y_pb_p=v_{p-1}b_p=v_{p-1}$, where the last equality is true because $v_{p-1}\in\ima(b_p)$.
Thus $c_{y_p}b_p=c_{v_{p-1}}$.
By Definition~\ref{dco1}, $y_p$ and $s$ belong to Class-1 of $\Ker(e_1)$, and $s\in\ima(e_1)$. It
follows that $c_{y_p}e_1=c_s$. Again by Definition~\ref{dco1}, $y_0$ and $y_{p-1}$ belong to Class-2 of
$\Ker(a_p)$, and $y_{p-1}\in\ima(a_p)$. Hence $c_{y_0}a_p=c_{y_{p-1}}$. Similarly,
$c_{y_0}b_p=c_{v_p}$ and $c_{y_0}e_1=c_{v_k}$.
By Lemma~\ref{lev2},
\begin{align}
c_{y_i}a_j&=(c_{y_i}a_i)a_j=c_{y_i}(a_ia_j)=c_{y_i}c_{y_{j-1}}=c_{y_{j-1}},\notag\\
c_{y_j}a_i&=(c_{y_j}a_j)a_i=c_{y_j}(a_ja_i)=c_{y_j}c_{y_i}=c_{y_i},\notag\\
c_{y_i}b_j&=(c_{y_i}a_i)b_j=c_{y_i}(a_ib_j)=c_{y_i}c_{v_j}=c_{v_j},\notag\\
c_{y_j}b_i&=(c_{y_j}a_j)b_i=c_{y_j}(a_jb_i)=c_{y_j}c_{v_{i-1}}=c_{v_{i-1}}.\notag
\end{align}
We have proved (1). Proofs of (2) and (3) are similar.
\end{proof}

Table~2 presents the Cayley table for $S_0^2$.

\begin{table}[h]
\[
\begin{tabular}{c|ccccccccccc}
     & $a_1$ &  $a_2$  & $b_1$ & $b_2$ & $e_1$ & $c_{y_0}$ & $c_{y_1}$ & $c_{y_2}$ & $c_{v_1}$ & $c_{v_2}$ & $c_s$ \\\hline
 $a_1$ & $a_1$ &  $c_{y_1}$  & $b_1$ & $c_{v_2}$ & $e_1$ & $c_{y_0}$ & $c_{y_1}$ & $c_{y_2}$ & $c_{v_1}$ & $c_{v_2}$ & $c_s$ \\
 $a_2$ & $c_{y_1}$ &  $a_2$  & $c_{y_2}$ & $b_2$ & $c_s$ & $c_{y_0}$ & $c_{y_1}$ & $c_{y_2}$ & $c_{v_1}$ & $c_{v_2}$ & $c_s$ \\
 $b_1$ & $a_1$ &  $c_{y_2}$  & $b_1$ & $c_{v_1}$ & $e_1$ & $c_{y_0}$ & $c_{y_1}$ & $c_{y_2}$ & $c_{v_1}$ & $c_{v_2}$ & $c_s$ \\
 $b_2$ & $c_{y_0}$ &  $a_2$  & $c_{v_1}$ & $b_2$ & $c_{v_2}$ & $c_{y_0}$ & $c_{y_1}$ & $c_{y_2}$ & $c_{v_1}$ & $c_{v_2}$ & $c_s$ \\
 $e_1$ & $a_1$ &  $c_{y_1}$  & $b_1$ & $c_{v_2}$ & $e_1$ & $c_{y_0}$ & $c_{y_1}$ & $c_{y_2}$ & $c_{v_1}$ & $c_{v_2}$ & $c_s$ \\
 $c_{y_0}$ & $c_{y_0}$ &  $c_{y_1}$  & $c_{v_1}$ & $c_{v_2}$ & $c_{v_2}$ & $c_{y_0}$ & $c_{y_1}$ & $c_{y_2}$ & $c_{v_1}$ & $c_{v_2}$ & $c_s$ \\
 $c_{y_1}$ & $c_{y_1}$ &  $c_{y_1}$  & $c_{y_2}$ & $c_{v_2}$ & $c_{s}$ & $c_{y_0}$ & $c_{y_1}$ & $c_{y_2}$ & $c_{v_1}$ & $c_{v_2}$ & $c_s$ \\
 $c_{y_2}$ & $c_{y_1}$ &  $c_{y_2}$  & $c_{y_2}$ & $c_{v_1}$ & $c_{s}$ & $c_{y_0}$ & $c_{y_1}$ & $c_{y_2}$ & $c_{v_1}$ & $c_{v_2}$ & $c_s$ \\
 $c_{v_1}$ & $c_{y_0}$ &  $c_{y_2}$  & $c_{v_1}$ & $c_{v_1}$ & $c_{v_2}$ & $c_{y_0}$ & $c_{y_1}$ & $c_{y_2}$ & $c_{v_1}$ & $c_{v_2}$ & $c_s$ \\
 $c_{v_2}$ & $c_{y_0}$ &  $c_{y_1}$  & $c_{v_1}$ & $c_{v_2}$ & $c_{v_2}$ & $c_{y_0}$ & $c_{y_1}$ & $c_{y_2}$ & $c_{v_1}$ & $c_{v_2}$ & $c_s$ \\
 $c_s$ & $c_{y_1}$ &  $c_{y_1}$  & $c_{y_2}$ & $c_{v_2}$ & $c_{s}$ & $c_{y_0}$ & $c_{y_1}$ & $c_{y_2}$ & $c_{v_1}$ & $c_{v_2}$ & $c_s$ \\
\end{tabular}
\]
\caption{Cayley table for $S_0^2$.}
\end{table}

\begin{lemma}\label{lev3}
Let $g,h,c_z\in S_0^k$ such that $c_z$ is a constant and $g-c_z-h$ is a path in $\cg(S_0^k)$. Then $gh=hg$.
\end{lemma}
\begin{proof}
Note that $g,h$ are not constants since different constants do not commute. Thus
$g$ and $h$ are generators from list (\ref{ecev2}). We may assume that $g$ is to the left of $h$
in the list. Since $c_z$ commutes with both $g$ and $h$, $z\in\ima(g)\cap\ima(h)$ by Lemma~\ref{lcon}.
Suppose $g=a_i$, where $1\leq i\leq k-1$. Then $h=a_{i+1}$ since $a_{i+1}$ is the only generator to the right of $a_i$
whose image is not disjoint from $\ima(a_i)$. Similarly, if $g=a_k$ then $h=b_1$; if $g=b_i$ ($1\leq i\leq k-1$) then $h=b_{i+1}$;
and if $g=b_k$ then $h=e_1$. Hence $gh=hg$ by Corollary~\ref{cev2a}.
\end{proof}

\begin{lemma}\label{lev4} The paths
\begin{itemize}
  \item [\rm(i)] $\tau_1=c_{y_0}-a_1-\cdots-a_k-b_1-\cdots-b_k-c_{v_k}$,
  \item [\rm(ii)] $\tau_2=c_{y_1}-a_2-\cdots-a_k-b_1-\cdots-b_k-e_1-c_s$
\end{itemize}
are the only minimal $l$-paths in $\cg(S_0^k)$ with constants as the endpoints.
\end{lemma}
\begin{proof}
We have that $\tau_1$ and $\tau_2$ are $l$-paths by Lemmas~\ref{lev2} and~\ref{lev2a}.
Suppose that $\lam=c_z-\cdots-c_w$ is a minimal $l$-path in $\cg(S_0^k)$
with constants $c_z$ and $c_w$ as the endpoints.
Recall that $z,w\in\{y_0,y_1,\ldots,y_k,v_1,\ldots,v_k,s\}$.
We may assume that $z$ is to the left of $w$ in the list $y_0,y_1,\ldots,y_k,v_1,\ldots,v_k,s$.
Since $\lam$ is minimal, Lemma~\ref{lev3} implies that $\lam$ does not contain any constants
except $c_z$ and $c_w$. There are five cases to consider.
\begin{itemize}
  \item [(a)] $\lam=c_{y_i}-\cdots-c_{y_j}$, where $0\leq i<j\leq k$.
  \item [(b)] $\lam=c_{y_i}-\cdots-c_{v_j}$, where $0\leq i\leq k$, $1\leq j\leq k$.
  \item [(c)] $\lam=c_{y_i}-\cdots-c_s$, where $0\leq i\leq k$.
  \item [(d)] $\lam=c_{v_i}-\cdots-c_{v_j}$, where $1\leq i<j\leq k$.
  \item [(e)] $\lam=c_{v_i}-\cdots-c_s$, where $1\leq i\leq k$.
\end{itemize}

Suppose (a) holds, that is, $\lam=c_{y_i}-\cdots-h-c_{y_j}$, $0\leq i<j\leq k$. Since $hc_{y_j}=c_{y_j}h$,
either $h=a_j$ or $h=a_{j+1}$ (where $a_{k+1}=b_1$) (since $a_j$ and $a_{j+1}$ are the only generators
that have $y_j$ in their image). Suppose $h=a_{j+1}$. Then, by Corollary~\ref{cev2a}, either
$\lam=c_{y_i}-\cdots-a_j-a_{j+1}-c_{y_j}$ or $\lam=c_{y_i}-\cdots-a_{j+2}-a_{j+1}-c_{y_j}$
(where $a_{j+2}=b_1$ if $j=k-1$, and $a_{j+2}=b_2$ if $j=k$). In the latter case,
\[
\lam=c_{y_i}-\cdots-a_1-e_1-b_k-\cdots-b_1-a_k-\cdots-a_{j+2}-a_{j+1}-c_{y_j},
\]
which is a contradiction since
$a_1$ and $e_1$ do not commute. Thus either $\lam=c_{y_i}-\cdots-a_j-c_{y_j}$
or $\lam=c_{y_i}-\cdots-a_j-a_{j+1}-c_{y_j}$. In either case, $\lam$ contains $a_j$, and so $c_{y_i}a_j=c_{y_j}a_j$
(since $\lam$ is an $l$-path). But, by Lemma~\ref{lev2a}, $c_{y_i}a_j=c_{y_{j-1}}$ and $c_{y_j}a_j=c_{y_j}$.
Hence $c_{y_{j-1}}=c_{y_j}$, which is a contradiction.

Suppose (b) holds, that is, $\lam=c_{y_i}-g-\cdots-h-c_{v_j}$, $0\leq i\leq k$ and $1\leq j\leq k$. Then $g$ is either $a_i$ or $a_{i+1}$
($g=a_{i+1}$ if $i=0$)
and $h$ is either $b_j$ or $b_{j+1}$ (where $b_{k+1}=e_1$). In any case,
$\lam=c_{y_i}-g-\cdots-a_k-b_1-\cdots-h-c_{v_j}$. Suppose $i\geq1$.
Then, by Lemma~\ref{lev2a} and the fact that
$\lam$ is an $l$-path, $c_{v_0}=c_{y_i}b_1=c_{v_j}b_1=c_{v_1}$, which is a contradiction.
If $i=0$ and $j<k$, then $c_{y_{k-1}}=c_{y_0}a_k=c_{v_j}a_k=c_{y_k}$, which is again a contradiction.
If $i=0$ and $j=k$, then $g=a_1$, and so $\lam=\tau_1$.

Suppose (c) holds, that is, $\lam=c_{y_i}-g-\cdots-a_k-b_1-\cdots-b_k-e_1-c_s$, $0\leq i\leq k$, where
$g$ is either $a_i$ or $a_{i+1}$ ($g=a_{i+1}$ if $i=0$). If $i>1$, then
$c_{v_{i-1}}=c_{y_i}b_i=c_sb_i=c_{v_i}$, which is a contradiction.
If $i=0$, then $c_{v_k}=c_{y_0}e_1=c_se_1=c_s$, which is a contradiction.
If $i=1$ and $g=a_1$, then $\lam$ is not minimal since $c_{y_1}-a_2$,
so $a_1$ can be removed. Finally, if $i=1$ and $g=a_2$, then $\lam=\tau_2$.

Suppose (d) holds, that is, $\lam=c_{v_i}-g-\cdots-h-c_{v_j}$, $1\leq i<j\leq k$, where $g$ is either $b_i$ or $b_{i+1}$
and $h$ is either $b_j$ or $b_{j+1}$ (where $b_{k+1}=e_1$). In any case,
$\lam$ contains $b_j$, and so $c_{v_{j-1}}=c_{v_i}b_j=c_{v_j}b_j=c_{v_j}$, which is a contradiction.

Suppose (e) holds, that is, $\lam=c_{v_i}-\cdots-e_1-c_s$, $1\leq i\leq k$.
Then $c_{v_k}=c_{v_i}e_1=c_se_1=c_s$, which is a contradiction.

We have exhausted all possibilities and obtained that $\lam$ must be equal to $\tau_1$ or $\tau_2$.
The result follows.
\end{proof}

\begin{lemma}\label{lev5}
The path  $\pi=a_1-\cdots-a_k-b_1-\cdots-b_k-e_1$ is a unique minimal $l$-path in $\cg(S_0^k)$
with at least one endpoint that is not a constant.
\end{lemma}
\begin{proof}
We have that $\pi$ is an $l$-path by Lemmas~\ref{lev2} and~\ref{lev2a}.
Suppose that $\lam=e-\cdots-f$ is a minimal $l$-path in $\cg(S_0^k)$
such that $e$ or $f$ is not a constant.

We claim that $\lam$ does not contain any constant $c_z$. By Lemma~\ref{lev3}, there is no
constant $c_z$ such that $\lam=e-\cdots-c_z-\cdots-f$ (since otherwise $\lam$ would not be minimal).
We may assume that $f$
is not a constant. But then $e$ is not a constant either since otherwise we would have that
$ef$ is a constant and $ff=f$ is not a constant. But this is impossible since $\lam$ is an $l$-path,
and so $ef=ff$. The claim has been proved.

Thus all elements in $\lam$ are generators from list (\ref{ecev2}). We may assume that $e$ is to
the left of $f$ (according to the ordering in (\ref{ecev2})). Since $\lam$ is an $l$-path,
$e=ee=fe$. Hence, by Lemma~\ref{lev2}, $e=a_p$ and $f=b_p$ (for some $p\in\{1,\ldots,k\}$)
or $e=b_1$ and $f=e_1$ or $e=a_1$ and $f=e_1$.

Suppose that $e=a_p$ and $f=b_p$ for some $p$. Then, by Corollary~\ref{cev2a},
$\lam=a_p-\cdots-a_k-b_1-\cdots-b_p$. (Note that $\lam=a_p-a_{p-1}-\cdots-a_1-e_1-b_k-\cdots-b_p$ is impossible
since $a_1e_1\ne e_1a_1$.) If $p>1$ then, by Lemma~\ref{lev2},
$c_{v_0}=a_pb_1=b_pb_1=c_{v_1}$, which is a contradiction. If $p=1$, then
$c_{y_{k-1}}=a_1a_k=b_1b_k=c_{y_k}$, which is again a contradiction.

Suppose that $e=b_1$ and $f=e_1$. Then $\lam=b_1-\cdots-b_k-e_1$, and so
$c_{v_{k-1}}=b_1b_k=e_1b_k=c_{v_k}$, which is a contradiction.

Hence we must have $e=a_1$ and $f=e_1$. But then, by Corollary~\ref{cev2a},
$\lam=a_1-\cdots-a_k-b_1-\cdots-b_k-e_1=\pi$. The result follows.
\end{proof}

\begin{theorem}\label{teve}
For every even integer $n\geq2$, there is a band $S$ with knit degree $n$.
\end{theorem}
\begin{proof}
Let $n=2$. Consider the band $S=\{a,b,c,d\}$ defined by the following Cayley table:
\[
\begin{tabular}{c|cccc}
     & $a$ &  $b$  & $c$ & $d$ \\\hline
 $a$ & $a$ &  $b$  & $c$ & $d$ \\
 $b$ & $b$ &  $b$  & $b$ & $b$ \\
 $c$ & $a$ &  $b$  & $c$ & $d$ \\
 $d$ & $d$ &  $d$  & $d$ & $d$ \\
 \end{tabular}
\]
It is easy to see that the center of $S$ is empty and $a-b-c$ is a shortest $l$-path in $\cg(S)$. Thus $\kd(S)=2$.

Let $n=2k$ where $k\geq2$. Consider the semigroup $S_0^k$ defined by (\ref{edco1}). Then, by Corollary~\ref{cev2},
$S_0^k$ is a band. The paths $\tau_1$, $\tau_2$, and $\pi$ from Lemmas~\ref{lev4} and \ref{lev5}
are the only minimal $l$-paths in $\cg(S_0^k)$.
Since $\tau_1$ has length $2k+1=n+1$, $\tau_2$ has length $2k+2=n+2$, and $\pi$ has length $2k=n$, it follows that $\kd(S_0^k)=n$.
\end{proof}

\subsection{The Odd Case}\label{ssodd}

Suppose $n=2k+1\geq5$ is odd. We will obtain a band $S$ of knit degree $n$ by slightly modifying
the construction of the band $S_0^k$ from Definition~\ref{dco1}. Recall that $S_0^k$
has knit degree $2k$ (see the proof of Theorem~\ref{teve}). We will obtain a band of knit degree $n=2k+1$
by simply removing transformations $e_1$ and $c_s$ from $S_0^k$.

\begin{defi}\label{dco2}
{\rm
Let $k\geq2$ be an integer. Consider the following subset of the semigroup $S^0_k$ from Definition~\ref{dco1}:
\begin{equation}\label{edco21}
S^1_k=S^0_k-\{e_1,c_s\}=\{a_1,\ldots,a_k,b_1,\ldots,b_k,c_{y_0},c_{y_1},\ldots,c_{y_k},c_{v_1},\ldots,c_{v_k}\}.
\end{equation}
By Lemmas~\ref{lev2} and \ref{lev2a}, $S^1_k$ is a subsemigroup of $S^0_k$.
}
\end{defi}

\begin{rem}\label{rdco2}
{\rm
Note that $r$ and $s$, which still occur in the domain (but not the image)
of each element of $S_1^k$, are now superfluous. We can remove them from the domain of each element of $S^1_k$
and view $S^1_k$ as a semigroup of transformations on the set
\[
X=\{y_0,y_1,\ldots,y_k=v_0,v_1,\ldots,v_k,x_1,\ldots,x_k,u_1,\ldots,u_k\}.
\]
}
\end{rem}

It is clear from the definition of $S_1^k$ that the multiplication table for $S_1^k$ is the multiplication
table for $S_0^k$ (see Lemmas~\ref{lev2} and \ref{lev2a})
with the rows and columns $e_1$ and $c_s$ removed. This new multiplication table is given by Lemmas~\ref{lev2}
and \ref{lev2a} if we ignore the multiplications involving $e_1$ or $c_s$. Therefore, the following lemma
follows immediately from Corollary~\ref{cev2} and Lemmas~\ref{lev4} and \ref{lev5}.

\begin{lemma}\label{levnew1} Let $S_1^k$ be the semigroups defined by {\rm(\ref{edco21})}. Then
$S_1^k$ is a band and $\tau=c_{y_0}-a_1-\cdots-a_k-b_1-\cdots-b_k-c_{v_k}$
is the only minimal $l$-path in $\cg(S_1^k)$.
\end{lemma}

\begin{theorem}\label{todd}
For every odd integer $n\geq5$, there is a band $S$ of knit degree $n$.
\end{theorem}
\begin{proof}
Let $n=2k+1$ where $k\geq2$. Consider the semigroup $S_1^k$ defined by (\ref{edco21}). Then, by Lemma~\ref{levnew1},
$S_1^k$ is a band and $\tau=c_{y_0}-a_1-\cdots-a_k-b_1-\cdots-b_k-c_{v_k}$
is the only minimal $l$-path in $\cg(S_1^k)$.
Since $\tau$ has length $2k+1=n$, it follows that $\kd(S_1^k)=n$.
\end{proof}

The case $n=3$ remains unresolved.
\vskip 2mm
\noindent{\bf Open Question.} Is there a semigroup of knit degree $3$?

\section{Commuting Graphs with Arbitrary Diameters}\label{sald}
\setcounter{equation}{0}
In Section~\ref{stx}, we showed that, except for some special cases, the commuting graph of any ideal of
the semigroup $T(X)$ has diameter $5$.
In this section, we use the constructions of Section~\ref{smlp} to show that
there are semigroups whose commuting graphs have any prescribed diameter.
We note that the situation is (might be) quite different in group theory:
it has been conjectured
that there is an upper bound for the diameters of the connected commuting graphs of finite non-abelian
groups \cite[Conjecture~2.2]{IrJa08}.

\begin{theorem}\label{tald}
For every $n\geq2$, there is a semigroup $S$ such that the diameter of $\cg(S)$ is $n$.
\end{theorem}
\begin{proof}
Let $n\in\{2,3,4\}$. The commuting graph of the band $S$ defined by the Cayley table in the proof of Theorem~\ref{teve}
is the cycle $a-b-c-d-a$. Thus the diameter of $\cg(S)$ is $2$.
Consider the semigroup $S$ defined by the following table:
\[
\begin{tabular}{r|rrrr}
 & a & b & c & d\\ \hline
    a & a & a & a & a \\
    b & a & b & c & c \\
    c & c & c & c & c \\
    d & c & d & c & c
\end{tabular} \hspace{.5cm}
\]
Note that $Z(S)=\emptyset$ and $\cg(S)$ is the chain $a-b-c-d$. Thus the diameter of $\cg(S)$
is $3$. The diameter of $\cg(J_4)$ is $4$ (where $J_4$ is an ideal of $T(X)$ with $|X|=5$).

Let $n\geq5$. Suppose $n$ is even.
Then $n=2k+2$ for some $k\geq2$. Consider the band $S_0^k$ from
Definition~\ref{dco1}. Since $c_{y_0}$ and $a_1$ are the only elements of $S_0^k$ whose image
contains $y_0$, they are the only elements of $S_0^k$ commuting with $c_{y_0}$ (see Lemma~\ref{lcon}).
Similarly, $e_1$ and $c_s$ are the only elements commuting with $c_s$. Therefore, it follows from
Corollary~\ref{cev2a} that $c_{y_0}-a_1-\cdots-a_k-b_1-\cdots-b_k-e_1-c_s$ is a shortest path in $\cg(S_0^k)$
from $c_{y_0}$ to $c_s$, that is, the distance between $c_{y_0}$ and $c_s$ is $2k+2=n$.
Since $a_1-\cdots-a_k-b_1-\cdots-b_k-e_1$ is a path in $\cg(S_0^k)$, $c_{y_i}a_i=a_ic_{y_i}$
and $c_{v_i}b_i=b_ic_{v_i}$ ($1\leq i\leq k$), it follows that the distance between any two vertices
of $\cg(S_0^k)$ is at most $2k+2$. Hence the diameter of $\cg(S_0^k)$ is $n$.

Suppose $n$ is odd.
Then $n=2k+1$ for some $k\geq2$. Consider the band $S_1^k$ from
Definition~\ref{dco2}. Then $c_{y_0}-a_1-\cdots-a_k-b_1-\cdots-b_k-c_{v_k}$ is a shortest path in $\cg(S_1^k)$
from $c_{y_0}$ to $c_{v_k}$, that is, the distance between $c_{y_0}$ and $c_{v_k}$ is $2k+1=n$.
As for $S_0^k$, we have $c_{y_i}a_i=a_ic_{y_i}$
and $c_{v_i}b_i=b_ic_{v_i}$ ($1\leq i\leq k$). Thus the distance between any two vertices
of $S_1^k$ is at most $2k+1$, and so the diameter of $\cg(S_1^k)$ is $n$.
\end{proof}

\section{Schein's Conjecture}\label{ssch}
\setcounter{equation}{0}
The results obtained in Section~\ref{smlp} enable us to settle a conjecture formulated by B.M. Schein
in 1978 \cite[p.~12]{Sc78}. Schein stated his conjecture in the context of the attempts to characterize
the $r$-semisimple bands.

A right congruence $\tau$ on a semigroup S is said to be modular if there exists an element $e\in S$ such that
$(ex)\tau x$ for all $x\in S$. The radical $R_r$ on a band $S$ is the intersection of all maximal modular right congruences
on $S$ \cite{Oe66}. A band $S$ is called \emph{$r$-semisimple} if its radical $R_r$ is the identity relation on $S$.

In 1969, B.D. Arendt announced a characterization of $r$-semisimple bands \cite[Theorem~18]{Ar69}.
In 1978, B.M Schein pointed out that Arendt's characterization is incorrect and proved \cite[p.~2]{Sc78}
that a band $S$ is $r$-semisimple if and only if it satisfies infinitely many quasi-identities:
(1) and $(A_n)$ for all integers $n\geq1$, where
\begin{align}
(1)\,\,\,\,\,&zx=zy\imp xy=yx,\notag\\
(A_n)\,\,\,\,\,&x_1x_2=x_2x_1\wedge x_2x_3=x_3x_2\wedge\ldots\wedge x_{n-1}x_n=x_nx_{n-1}\wedge\notag\\
&\wedge x_1x_1=x_nx_1\wedge x_1x_2=x_nx_2\wedge\ldots\wedge x_1x_n=x_nx_n \imp x_1=x_n.\notag
\end{align}
Schein observed that $(A_1)$ and $(A_2)$ are true in every band, that $(A_3)$ easily follows from (1),
and that Arendt's characterization of $r$-semisimple bands is equivalent to (1). He used the last observation
to show that Arendt's characterization is incorrect by providing an example of a band $T$ for which
(1) holds but $(A_4)$ does not. We note that Schein's example is incorrect since the Cayley table
in \cite[p.~10]{Sc78}, which is supposed to define $T$, does not define a semigroup because the
operation is not associative: $(4*1)*1=10\neq8=4*(1*1)$. However, Schein was right that
it is not true that condition (1) implies $(A_n)$ for all $n$. The semigroup $S_0^2$ (see Table~2)
satisfies (1) but it does not satisfy $(A_5)$ since $a_1-a_2-b_1-b_2-e_1$ is an $l$-path
(so the premise of $(A_5)$ holds) but $a_1\ne e_1$.

At the end of the paper, Schein formulates his conjecture \cite[p.~12]{Sc78}:
\vskip 2mm
\noindent{\bf Schein's Conjecture.} For every $n>1$, $(A_n)$ does not imply $(A_{n+1})$.
\vskip 2mm
The reason that Section~\ref{smlp} enables us to settle Schein's conjecture is the following lemma.

\begin{lemma}\label{lscon}
Let $n\geq1$ and let $S$ be a band with no central elements. Then $S$ satisfies $(A_n)$
if and only if $\cg(S)$ has no $l$-path of length $<n$.
\end{lemma}

\begin{proof}
First note that $(A_n)$ can be expressed as: for all $x_1,\ldots,x_n\in S$,
\begin{equation}\label{elscon}
x_1-\cdots-x_n\mbox{ and }x_1x_i=x_nx_i\mbox{ $(1\leq i\leq n)$}\imp x_1=x_n.
\end{equation}
(Here, we allow $x-x$ and do not require that $x_1,\ldots,x_n$ be distinct.)

Assume $S$ satisfies $(A_n)$. Suppose to the contrary that $\cg(S)$ has an $l$-path
$\lam=x_1-\cdots-x_k$ of length $<n$, that is, $k\leq n$. Then
$x_1-\cdots-x_k-x_{k+1}-\cdots-x_n$, where $x_i=x_k$ for every $i\in\{k+1,\ldots,n\}$,
and so $x_1=x_n=x_k$ by (\ref{elscon}). This is a contradiction since $\lam$ is a path.

Conversely, suppose that $\cg(S)$ has no $l$-path of length $<n$. Let
$x_1-\cdots-x_n$ and $x_1x_i=x_nx_i$ ($1\leq i\leq n$). Suppose to the contrary
that $x_1\ne x_n$.
If there are $i$ and $j$ such that
$1\leq i<j\leq n$ and $x_i=x_j$, we can replace $x_1-\cdots-x_i-\cdots-x_j-\cdots-x_n$
with $x_1-\cdots-x_i-x_{j+1}-\cdots-x_n$. Therefore, we can assume that $x_1,\ldots,x_n$
are pairwise distinct. Recall that
$S$ has no central elements, so all $x_i$ are vertices in $\cg(S)$. Thus
$x_1-\cdots-x_n$ is an $l$-path in $\cg(S)$ of length $n-1$, which is a contradiction.
\end{proof}

First, Schein's conjecture is false for $n=3$.

\begin{prop}\label{pscon1} $(A_3)\imp(A_4)$.
\end{prop}
\begin{proof}
Suppose a band $S$ satisfies $(A_3)$, that is,
\begin{equation}\label{essch1}
x_1x_2=x_2x_1 \wedge x_2x_3=x_3x_2 \wedge x_1x_1=x_3x_1 \wedge x_1x_2=x_3x_2 \wedge x_1x_3=x_3x_3 \imp x_1=x_3.
\end{equation}
To prove that $S$ satisfies $(A_4)$, suppose that
\[
y_1y_2=y_2y_1 \wedge y_2y_3=y_3y_2 \wedge y_3y_4=y_4y_3 \wedge y_1y_1=y_4y_1 \wedge y_1y_2=y_4y_2 \wedge y_1y_3=y_4y_3 \wedge y_1y_4=y_4y_4.
\]
Take $x_1=y_1$, $x_2=y_2y_3$, and $x_3=y_4$. Then $x_1,x_2,x_3$ satisfy the premise of (\ref{essch1}):
\begin{align}
x_1x_2&=y_1y_2y_3=y_1y_3y_2=y_4y_3y_2=y_3y_4y_2=y_3y_1y_2=y_3y_2y_1=y_2y_3y_1=x_2x_1,\notag\\
x_2x_3&=y_2y_3y_4=y_2y_4y_3=y_2y_1y_3=y_1y_2y_3=y_4y_2y_3=x_3x_2,\notag\\
x_1x_1&=y_1y_1=y_4y_1=x_3x_1,\,x_1x_2=y_1y_2y_3=y_4y_2y_3=x_3x_2,\,x_1x_3=y_1y_4=y_4y_4=x_3x_3.\notag
\end{align}
Thus, by (\ref{essch1}), $y_1=x_1=x_3=y_4$, and so $(A_4)$ holds.
\end{proof}

Second, Schein's conjecture is true for $n\ne3$.

\begin{prop}\label{pscon2}
If $n>1$ and $n\ne3$, then $(A_n)$ does not imply $(A_{n+1})$.
\end{prop}
\begin{proof}
Consider the band $S=\{e,f,0\}$, where $0$ is the zero, $ef=f$, and $fe=e$. Then $e-0-f$,
$ee=fe$, $e0=f0$, $ef=ff$, and $e\ne f$. Thus $S$ does not satisfy $(A_3)$. But $S$ satisfies $(A_2)$
since $(A_2)$ is true in every band. Hence $(A_2)$ does not imply $(A_3)$.

Let $n\geq4$. Then, by Theorems~\ref{teve} and \ref{todd} and their proofs, the band $S$
constructed in Definition~\ref{dco1} (if $n$ is even) or Definition~\ref{dco2} (if $n$ is odd)
has knit degree $n$. By Lemmas~\ref{lev2} and \ref{lev2a}, $S$ has no central elements.
Since $\kd(S)=n$, there is an $l$-path in $\cg(S)$
of length $n$ and there is no $l$-path in $\cg(S)$ of length $<n$. Hence, by Lemma~\ref{lscon},
$S$ satisfies $(A_n)$ and $S$ does not satisfy $(A_{n+1})$. Thus $(A_n)$
does not imply $(A_{n+1})$.
\end{proof}

\section{Problems}\label{spro}
We finish this paper with a list of some problems concerning commuting graphs of semigroups.
\begin{itemize}
  \item[(1)] Is there a semigroup with knit degree 3? Our guess is that such a semigroup does not exist.
  \item[(2)] Classify the semigroups whose commuting graph is eulerian (proposed by M. Volkov). The same problem for hamiltonian and  planar graphs.
  \item[(3)] Classify the commuting graphs of semigroups.
  \item[(4)] Is it true that for all natural numbers $n\geq 3$, there is a
semigroup $S$ such that the clique number (girth, chromatic number) of $\cg(S)$ is $n$?
  \item[(5)] Classify the semigroups $S$ such that the clique and chromatic numbers of $\cg(S)$ coincide.
  \item[(6)] Calculate the clique and chromatic numbers of the commuting graphs of $T(X)$ and $\fend(V)$, where
$X$ is a finite set and $V$ is a finite-dimensional vector space over a finite field.
\item[(7)] Let $\cg(S)$ be the commuting graph of a finite non-commutative semigroup $S$.
An \emph{$\mrl$-path} is a path $a_1-\cdots-a_m$ in $\cg(S)$ such that $a_1\ne a_m$ and
$a_1a_ia_1=a_ma_ia_m$ for all $i=1,\ldots,m$. For $\mrl$-paths, prove the results analogous to the results
for $l$-paths contained in this paper.
\item[(8)]
Find classes of finite non-commutative semigroups
such that if $S$ and $T$ are two semigroups in that class and $\cg(S)\cong \cg(T)$, then $S\cong T$.
\end{itemize}

\section{Acknowledgments}
We are pleased to acknowledge the assistance of the automated deduction tool
\textsc{Prover9} and the finite model builder \textsc{Mace4}, both developed by
W. McCune \cite{McCune}. We also thank the developers of GAP \cite{Scel92}, L.H. Soicher for GRAPE \cite{So06},
and Aedan Pope and Kyle Pula for their suggestions after carefully reading the manuscript.

The first author was partially supported by FCT and FEDER, 
Project POCTI-ISFL-1-143 of Centro de Algebra da Universidade de Lisboa, by FCT and PIDDAC through the project PTDC/MAT/69514/2006, by 
PTDC/MAT/69514/2006 Semigroups and Languages, and by\newline PTDC/MAT/101993/2008 Computations in groups and semigroups.

\end{document}